\documentclass[10pt]{article}
\usepackage{amsmath,amsthm,amssymb,amsfonts}
\usepackage[ruled, vlined]{algorithm2e}
\usepackage{float}
\usepackage{graphicx}
\usepackage{booktabs}
\usepackage{capt-of}
\usepackage{doi}
\usepackage{multirow}
\usepackage{rotating}
\usepackage{tikz}
\usetikzlibrary{arrows}
\usepackage{caption}
\usepackage{siunitx}
\usepackage{subcaption}
\usepackage[section]{placeins}

\floatstyle{ruled}
\newfloat{algorithm}{htb}{alg}
\floatname{algorithm}{Algorithm}

\theoremstyle{plain}

\newtheorem{lemma}{Lemma}
\newtheorem{proposition}{Proposition}

\theoremstyle{definition}

\newtheorem{example}{Example}

\newcommand{\ints}{\mathbb{Z}}

\newcommand{\reals}{\mathbb{R}}

\newcommand{\vect}[1]{\mbox{\boldmath$#1$}}

\DeclareMathOperator{\val}{val}
\DeclareMathOperator{\cl}{cl}
\DeclareMathOperator{\argmin}{arg\,min}

\begin{document}

\title{Scheduling Network Maintenance Jobs with Release Dates and Deadlines to Maximize Total Flow Over Time: Bounds and Solution Strategies}

\author{Natashia Boland \qquad Thomas Kalinowski \qquad Simranjit Kaur\\[.5ex]
{\small University of Newcastle, Australia}}


\maketitle

\begin{abstract}
We consider a problem that marries network flows and scheduling, motivated by the need to schedule maintenance activities in infrastructure networks, such as rail or general logistics networks. Network elements must undergo regular preventive maintenance, shutting down the arc for the duration of the activity. Careful coordination of these arc maintenance jobs can dramatically reduce the impact of such shutdown jobs on the flow carried by the network. Scheduling such jobs between given release dates and deadlines so as to maximize the total flow over time presents an intriguing case to study the role of time discretization. Here we prove that if the problem data is integer, and no flow can be stored at nodes, we can restrict attention to integer job start times. However if flow can be stored, fractional start times may be needed. This makes traditional strong integer programming scheduling models difficult to apply. Here we formulate an exact integer programming model for the continuous time problem, as well as integer programming models based on time discretization that can provide dual bounds, and that can - with minor modifications - also yield primal bounds. The resulting bounds are demonstrated to have small gaps on test instances, and offer a good trade-off for bound quality against computing time.  

\end{abstract}

\section{Introduction}\label{sec:intro}
We consider a problem of scheduling maintenance jobs on the arcs of a network, where each job has a
given processing time, release date and deadline. The network is used to transport flow from a
source node to a sink node using capacitated arcs, where the capacity of an arc represents the
highest rate in flow units per unit time that it can carry. While a maintenance job is in progress
on an arc, the arc cannot carry flow, and the objective is to schedule the maintenance jobs so as to
maximize the total flow that can be transported by the network over a given planning time horizon. 

This problem is a natural marriage of scheduling with network flow: it unites the diverse and well established field of scheduling (see any of the many excellent texts available on the topic, e.g.~\cite{bakertrietsch2009scheduling,brucker2007scheduling,bruckerknust2012scheduling,conway2003scheduling,leung2004scheduling,pinedo2008scheduling,robert2010scheduling}) with dynamic network flows, which have been the subject of intense study in recent years, e.g.~\cite{koch2011flowsovertime,kotnyek2003annotated,skutella2009introduction}. Yet the problem was only recently introduced, first appearing in~\cite{boland2012scheduling}, where meta-heuristic approaches were explored. It was motivated by a study of a bulk goods export supply chain~\cite{boland2012mixed,boland2011optimizing_hvcc}, in which maintenance jobs on sections of the rail network and equipment in the export terminals were scheduled so as to maximize the throughput of the system. Strong NP-hardness of the problem is established in~\cite{boland2012scheduling} using a reduction 
to a network with only a single transshipment node (node other than the source or sink). The complexity of a variety of special cases is investigated in~\cite{boland2013unit_time}
where it is shown (among other results) that even in the case that all jobs have unit processing times and do not have release dates or deadlines, and that all transshipment nodes have equal inbound and outbound capacity, the problem is still strongly NP-hard. 

The problem exhibits a rich structure, making it attractive for complexity analysis, approximation, combinatorial algorithms, integer programming, and heuristics. It represents a natural extension to existing network models, and admits many interesting variants. For example, Tawarmalani and Li~\cite{Tawa}, motivated by a problem in highway maintenance, consider a multicommodity flow variant, 
providing complexity results, combinatorial algorithms, and integer programming models. The only other work combining network flow and scheduling that we are aware of is that of Nurre et al.~\cite{nurre2012restoring}, who schedule arc restoration tasks in the wake of a major disruption so as to maximize weighted flow over time. Whilst the problem does involve scheduling over a (weighted) maximum flow objective, there are several key differences to the problem we study here. In the problem of~\cite{nurre2012restoring}, the arc is closed from the start of the planning horizon until its restoration task is completed, from which time it is always open. The restoration task for each arc must be scheduled for completion by a work group, leading to a parallel machine scheduling structure absent in our case. 

All work to date on the problem we study here has ignored the possibility of flow storage at nodes. The latter is an important feature of real applications: in the export supply chain studied in~\cite{boland2012mixed}, the port terminal has stockyards for holding stockpiled material, which can provide outbound flow from the system even while maintenance shuts down inbound flow. The optimal schedule in the case that storage is ignored is likely to be far from optimal in the presence of storage. (We illustrate this point using the example given in Section~\ref{subsec:example}.)
Furthermore, the presence of storage has a fundamental effect on the properties of the problem: when no storage is allowed, integer data implies an optimal solution with integer job start times, but with storage, non-integer job start times may be required for optimality.

This makes the problem with storage a fascinating setting in which to study the role of time
discretization in integer programming models. It motivates questions such as what potential start times are important to consider? Can we restrict attention to particular times, or must all times in the continuum be considered in order to prove optimality? What kinds of discretization and what models can guarantee valid bounds on the optimal solution? In addressing these questions, this paper can be viewed as contributing to a rapidly increasing body of work exploring exact models based on coarser time discretizations, for example big bucket models in lot-sizing~\cite{wolsey2006production} and discretizations based on job release dates and deadlines in machine scheduling~\cite{SadykovBaptiste}, as well as approximate integer programs, 
for example for traveling salesman problems with time windows~\cite{dash2012timeBB,WangRegan2009} and for solving continuous-time dynamic network flows~\cite{HashemiNasrabadi2012}.

The prior work on close variants of the problem we study here either treats the selection of job start times heuristically, or assumes integer start times. In order to tackle a real problem in which maintenance jobs must be timed to within 15-minute intervals over a planning horizon of a year, the integer programming model presented in~\cite{boland2012mixed} is formulated in terms of a sparse set of possible start times selected heuristically for each job. The problem variants tackled in~\cite{boland2013unit_time,boland2012scheduling}
(all without storage) simply state the problem as one in which jobs must start at integer times. Whilst this property of start times is intuitively reasonable in the case without storage, it has not yet been formally proved; the first such proof is one of the key contributions of this paper. These are summarized as follows.
\begin{enumerate}
\item We give the first mixed integer linear programming model that solves the continuous time problem (with or without storage) exactly.
\item We provide the first formal proof that when no storage is allowed, there exists an optimal
  solution in which all job start times expressed in the form $a+\varepsilon_1
  b_1+\varepsilon_2b_2+\cdots+\varepsilon_kb_k$ where $a$ is a release date or a deadline of some
  job, $\varepsilon_i\in\{1,-1\}$ for all $i\in\{1,\ldots,k\}$, $b_i$ is the processing time of some
  job for all $i\in\{1,\ldots,k\}$, and the number of terms $k$ is less than twice the total number of jobs. This yields the property that if the data is integer, an optimal solution with integer start times is assured.
\item We demonstrate that when storage {\em is} allowed, even with integer data, non-integer job start times may be needed in an optimal solution. We prove that in this case we can (without loss of generality) restrict attention to {\em rational} start times, and in particular show that there exist optimal start times that can be expressed as rational numbers with a denominator that is {\em independent} of the job parameters (processing times, release dates and deadlines).
\item As a consequence, the time indexed formulations so useful in machine scheduling (e.g.~\cite{savelsbergh2005TImodel}) cannot be applied to this problem directly while assuring optimality.  Noting that the exact formulation is difficult to solve in practice, we provide the first integer programming model to give valid {\em dual} bounds, (upper bounds), based on discretization of time. The model can employ {\em any} discretization. 
    This property is particularly attractive, as it permits models to be time-scale invariant. For example, using the release dates and deadlines of jobs to form the time discretization yields a model with a size that is invariant to the length of the time horizon, depending only on the number of jobs; a unit time discretization based model's size would increase with increasing time horizon as well as with the number of jobs.
 \item We observe that the same model employing a discretization that conforms to the job parameters
   (release dates and deadlines are included, as is any time point in the discretization plus or
   minus any job processing time that stays between the job's release date and deadline) and with
   the additional restriction that jobs start at the start of a time interval, also provides
   feasible solutions when solved as an integer program (IP). Furthermore, if the time discretization is sufficiently fine, this formulation yields optimal solutions.
 \item This allows us to compute primal (lower) bounds for test instances with integer data by solving an IP based on the unit time discretization. We compare this approach with several heuristics that ``repair'' solutions found in the process of solving the integer programs that yield dual bounds. Our computational tests demonstrate the strength of both lower and upper bounds; we discuss the trade-offs between solution time and quality, and show that for the best approaches, optimality gaps are typically very small.
\end{enumerate}

\section{Problem formulation}\label{sec:problem}
Let $N=(V,A,s,t)$ be a network with node set $V$, arc set $A$, source node $s\in V$, and sink node $t\in V$. There is a  set $W\subseteq V\setminus\{s,\,t\}$ of storage nodes, and the capacity vector $\vect u=(u_x)_{x\in A\cup W}$ collects the arc capacities and the storage capacities of the nodes in $W$. For a node $v$, let $\delta^{\text{out}}(v)$ and $\delta^{\text{in}}(v)$ denote the set of arcs starting at $v$ and ending in $v$, respectively. In addition, we are given a set $A_1 \subseteq A$ of arcs that need to be shut for a maintenance job to be done without any preemptions. For simplicity, assume that for each arc $a\in A_1$ there is exactly one maintenance job, which is specified by its processing time $p_a$, its release time $r_a$, and its deadline $d_a$.  We consider this network over a time horizon $T$. The problem is to schedule the jobs, hence a feasible solution is a vector $\vect{t^*}=(t^*_{a})_{a\in A_1 }$ of start times with $t^*_{a}\in[r_{a},d_{a}-p_{a}]$ for all $a\in A_1$. Let $X$ denote the set of all feasible solutions, i.e. $X=\prod_{a\in A_1}[r_{a},d_{a}-p_{a}]$.

A feasible solution $\vect{t^*}$ is evaluated as follows. Let $0=t_0<t_1<\cdots<t_n=T$ be the increasing sequence obtained by ordering the set
\[\left\{t^*_{a}\ :\ a\in A_1\right\}\cup\left\{t^*_{a}+p_{a}\ :\ a\in
  A_1\right\}\cup\left\{0,T\right\}.\]
Note that the same time point $t$ can occur multiple times in this union, since it is possible that
several jobs start and end at time $t$. During each of the time intervals $[t_{i-1},t_i)$, $i=1,\dots,n$, thus induced by $t^*$, the state
of the network is constant: no maintenance job either starts or ends within the interval. 

For $a\in A_1$, let $I_a$ denote the set of intervals in which arc $a$ is shut for maintenance, i.e., $I_a=\{i\in\{1,\ldots,n\}\ :\ t^*_{a}<t_i\leqslant t^*_{a}+p_{a}\}$.
The value of the solution $\vect{t^*}$, denoted by $\val(\vect{t^*})$, can be characterized as the
optimal value of a maximum flow problem in a time-expanded network. For $a\in A$ and $i\in\{1,\ldots,n\}$, let $x_{ai}$ be the
flow on arc $a$ in time interval $[t_{i-1},t_i)$. For $v\in W$ and $i\in\{1,\ldots,n\}$, let
$x_{vi}$ be the amount of flow that is stored in $v$ at time $t_i$. We impose the boundary
conditions that the storage nodes are empty in the beginning and in the end of the time horizon,
i.e., $x_{v0}=x_{vn}=0$ for all $v\in W$. Finally, $\val(\vect{t^*})$ is the optimal objective
value of the following problem. 
\begin{align}
\text{maximize
}\sum_{i=1}^n&\left(\sum_{a\in\delta^{\text{out}}(s)}x_{ai}-\sum_{a\in\delta^{\text{in}}(s)}x_{ai}\right) \label{eq:objective_orig}\\
\text{s.t.}\quad \sum_{a\in\delta^{\text{out}}(v)}x_{ai} &= \sum_{a\in\delta^{\text{in}}(v)}x_{ai} && i\in\{1,\ldots,n\},\ v\in V\setminus (W\cup\{s,t\}), \label{eq:flowcon_1_orig}\\
\sum_{a\in\delta^{\text{out}}(v)}x_{ai}+x_{vi} &= \sum_{a\in\delta^{\text{in}}(v)}x_{ai}+x_{v,i-1} && i\in\{1,\ldots,n\},\ v\in W, \label{eq:flowcon_2_orig}\\
x_{ai} &\leqslant (t_i-t_{i-1})u_a && a\in A\setminus A_1,\ i\in\{1,\ldots,n\},\label{eq:arc_cap_orig1}\\
x_{ai} &\leqslant (t_i-t_{i-1})u_a && a\in A_1,\ i\in\{1,\ldots,n\}\setminus I_a,\label{eq:arc_cap_orig}\\
x_{ai} &=0 && a\in A_1,\ i\in I_a, \label{eq:outages_orig}\\
x_{vi} &\leqslant u_v && v\in W, i\in\{0,1,\ldots,n\}\label{eq:node_cap_orig}\\
x_{v0} &= x_{vn} =0 && v\in W,\label{eq:boundary_orig}\\
x_{ai},\, x_{vi} &\geqslant 0 && a\in A,\ v\in W,\ i\in\{1,\ldots,n\}.\label{eq:domains_orig}
\end{align}
The objective function~\eqref{eq:objective_orig} is the total throughput, i.e., the sum of the flow values over all time
periods. Constraints~\eqref{eq:flowcon_1_orig} and~\eqref{eq:flowcon_2_orig} are flow conservation
constraints for non-storage nodes and storage nodes, respectively. The incoming flow of a storage
node $v\in W$ in time period $i$, i.e., in the time interval $[t_{i-1},t_i)$, is the sum of the flow
that arrives at node $v$ in time period $i$ and the flow that that is stored in node $v$ at time
$t_{i-1}$. Similarly, the outgoing flow of a storage node $v\in W$ in time period $i$ is the sum of
the flow that arrives at node $v$ in this period and the flow that that is stored in node $v$ at
time $t_{i}$. Constraints~\eqref{eq:arc_cap_orig1}, \eqref{eq:arc_cap_orig}
and~\eqref{eq:outages_orig} are arc capacity constraints, where~\eqref{eq:outages_orig} captures the
arc outages. Constraints~\eqref{eq:node_cap_orig} are storage capacity constraints, and
constraints~\eqref{eq:boundary_orig} capture the boundary conditions. Our optimization problem is to find a start time vector $\vect{t^*}$ to maximize the total throughput:
\begin{equation}\label{eq:master_problem}
\varphi^*=\max\{\val(\vect{t^*})\ :\ \vect{t^*}\in X\}.
\end{equation}
Note that we may assume that in each interval $[t_{i-1},t_i)$ the flow rates for all arcs are
constant: on arc $a$, we have a flow rate of $x_{ai}/(t_i-t_{i-1})$ units of flow per time unit.

\subsection{An illustrative example}\label{subsec:example}

Both to illustrate the problem, and to demonstrate the importance of storage in finding an optimal maintenance schedule, we consider the  network and arc maintenance jobs given in Figure~\ref{fig:Example_1}, where arc labels indicate arc names and capacities (in parentheses). In this example, the time horizon is $T=3$, and only the job on arc $a$ needs to be scheduled: the job on arc $b$ must start at time $t^*_b=0$. Let $\hat{t}$ denote the start time of the job on arc $a$. From job $a$'s parameters, we see that $\hat{t}\in [0,1]$. As a consequence, there can only be {\em one} sequence of times $0 = t_0 \leqslant t_1 \leqslant t_2 \leqslant \dots \leqslant t_n = T$ resulting from ordering $\{\hat{t},t^*_b\}\cup\{\hat{t}+p_a,t^*_b+p_b\}\cup\{0,T\} = \{0,\hat{t},1,\hat{t}+2,3\}$, (where we have relaxed the strict inequalities to permit the job on $a$ to start at $t_0=0$ or at $t^*_b + p_b = 1$, the end time of the job on $b$, or to end at time $T=3$), and that is given by
\[ 0 = t_0 \leqslant t_1 = \hat{t} \leqslant 1 = t_2 < t_3 = \hat{t} + 2 \leqslant t_4= T = 3.\]
\begin{figure}[htb]
\begin{minipage}[b]{.6\textwidth}
  \begin{center}
 \begin{tikzpicture}[->,>=stealth',shorten >=1pt,auto,node distance=3cm,
   thick,main node/.style={circle,draw,font=\normalfont}]

   \node[main node] (1) {$s$};
   \node[main node] (2) [right of=1] {$v$};
   \node[main node] (3) [ right of=2] {$t$};

   \path[every node/.style={font=\normalfont}]
     (1) edge node{$a\ (2)$} (2)
     (2) edge node {$b\ (1)$} (3);
 \end{tikzpicture}
  \caption*{Network}
  \end{center}
\end{minipage}\hfill
\begin{minipage}[b]{.36\textwidth}
\setlength{\tabcolsep}{10pt}
\centering
\begin{tabular}{@{}ccccc@{}} \toprule
	arc & $r$ & $d$ & $p$ \\ \midrule
    $a$ & $0$ & $3$ & $2$ \\
    $b$ & $0$ & $1$ & $1$ \\
    \bottomrule
  \end{tabular}
\caption*{Job parameters}
\end{minipage}
\caption{Example to show the effect of storage capacity on the optimal maintenance schedule.}\label{fig:Example_1}
\end{figure}
The four time intervals thus induced have duration $\hat{t}$,$1-\hat{t}$, $\hat{t}+1$, $1-\hat{t}$,
respectively, and the corresponding flow networks for no storage and for storage at $v$ with
capacity 2 are illustrated in Figure~\ref{fig:example_networks}. 
\begin{figure}[htb]
\begin{minipage}[b]{.45\textwidth}
  \begin{center}
 \begin{tikzpicture}[->,>=stealth',shorten >=1pt,auto,node distance=2.5cm,
   thick,main node/.style={circle,draw,font=\normalfont}]

   \node[main node] (1) {$s$};
   \node[main node] (2) [right of=1] {$v$};
   \node[main node] (3) [right of=2] {$t$};
   \node (1a) [left of=1,node distance =.5cm,anchor=east] {$[0,\,\hat t)$};
   
   \node[main node] (4) [below of = 1,node distance = 1.5cm] {$s$};
   \node[main node] (5) [below of = 2,node distance = 1.5cm] {$v$};
   \node[main node] (6) [below of = 3,node distance = 1.5cm] {$t$};
   \node (4a) [left of=4,node distance =.5cm,anchor=east] {$[\hat t,\,1)$};
   
   \node[main node] (7) [below of = 4,node distance = 1.5cm] {$s$};
   \node[main node] (8) [below of = 5,node distance = 1.5cm] {$v$};
   \node[main node] (9) [below of = 6,node distance = 1.5cm] {$t$};
   \node (7a) [left of=7,node distance =.5cm,anchor=east] {$[1,\,\hat t+2)$};

   \node[main node] (10) [below of = 7,node distance = 1.5cm] {$s$};
   \node[main node] (11) [below of = 8,node distance = 1.5cm] {$v$};
   \node[main node] (12) [below of = 9,node distance = 1.5cm] {$t$};
   \node (10a) [left of=10,node distance =.5cm,anchor=east] {$[\hat t+2,\,3)$};

   \path[every node/.style={font=\normalfont}]
     (1) edge node{$2\hat t$} (2)
     (2) edge node {$0$} (3);
   \path[every node/.style={font=\normalfont}]
     (4) edge node{$0$} (5)
     (5) edge node {$0$} (6);
   \path[every node/.style={font=\normalfont}]
     (7) edge node{$0$} (8)
     (8) edge node {$\hat t+1$} (9);
   \path[every node/.style={font=\normalfont}]
     (10) edge node{$2-2\hat t$} (11)
     (11) edge node {$1-\hat t$} (12);
 \end{tikzpicture}
  \end{center}
\end{minipage}\hfill
\begin{minipage}[b]{.45\textwidth}
  \begin{center}
 \begin{tikzpicture}[->,>=stealth',shorten >=1pt,auto,node distance=2.5cm,
   thick,main node/.style={circle,draw,font=\normalfont}]

   \node[main node] (1) {$s$};
   \node[main node] (2) [right of=1] {$v$};
   \node[main node] (3) [right of=2] {$t$};
   \node (1a) [left of=1,node distance =.5cm,anchor=east] {$[0,\,\hat t)$};
   
   \node[main node] (4) [below of = 1,node distance = 1.5cm] {$s$};
   \node[main node] (5) [below of = 2,node distance = 1.5cm] {$v$};
   \node[main node] (6) [below of = 3,node distance = 1.5cm] {$t$};
   \node (4a) [left of=4,node distance =.5cm,anchor=east] {$[\hat t,\,1)$};
   
   \node[main node] (7) [below of = 4,node distance = 1.5cm] {$s$};
   \node[main node] (8) [below of = 5,node distance = 1.5cm] {$v$};
   \node[main node] (9) [below of = 6,node distance = 1.5cm] {$t$};
   \node (7a) [left of=7,node distance =.5cm,anchor=east] {$[1,\,\hat t+2)$};

   \node[main node] (10) [below of = 7,node distance = 1.5cm] {$s$};
   \node[main node] (11) [below of = 8,node distance = 1.5cm] {$v$};
   \node[main node] (12) [below of = 9,node distance = 1.5cm] {$t$};
   \node (10a) [left of=10,node distance =.5cm,anchor=east] {$[\hat t+2,\,3)$};

   \path[every node/.style={font=\normalfont}]
     (1) edge node{$2\hat t$} (2)
     (2) edge node {$0$} (3);
   \path[every node/.style={font=\normalfont}]
     (4) edge node{$0$} (5)
     (5) edge node {$0$} (6);
   \path[every node/.style={font=\normalfont}]
     (7) edge node{$0$} (8)
     (8) edge node {$\hat t+1$} (9);
   \path[every node/.style={font=\normalfont}]
     (10) edge node{$2-2\hat t$} (11)
     (11) edge node {$1-\hat t$} (12);
   \path[every node/.style={font=\normalfont}]
     (2) edge node{$2$} (5)
     (5) edge node {$2$} (8)
     (8) edge node {$2$} (11);
 \end{tikzpicture}
  \end{center}
\end{minipage}
\caption{Time expanded networks without storage (left) and with storage (right).  The capacity of each arc during the time interval indicated at the left of the network is written above the arc. }\label{fig:example_networks}   
\end{figure}
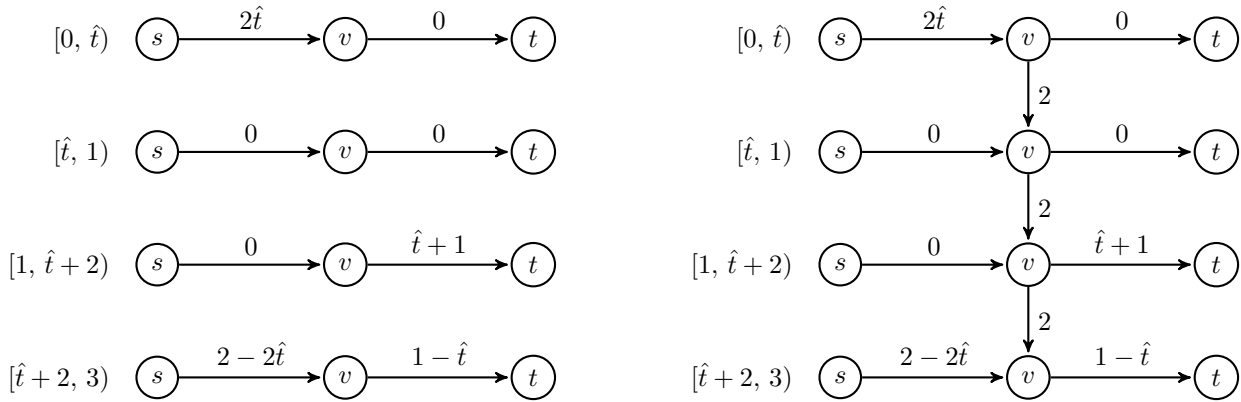
For the no-storage case the maximum throughput is
\[0+0+0+\min\{2-2\hat{t},\,1-\hat{t}\} = 1-\hat{t}.\]
This is maximized for $\hat{t}=0$, i.e., the optimal schedule is to take $t^*_a=t^*_b= 0$, which
gives a total throughput of $1$. Note that setting $\hat{t}=1$ gives a solution with zero throughput.
For the case with storage, the total capacity of arc $b$ gives an upper bound of $(\hat t+1)+(1-\hat
t)=2$ for the throughput, and this bound can be achieved only if arc $b$ is at capacity in each time
period. Therefore it is necessary that $2\hat t\geqslant\hat t+1$, i.e., $\hat t\geqslant
1$. This is also sufficient, since for $\hat t=1$ we get a feasible solution of value $2$ which is
shown on the left in Figure~\ref{fig:flow}. The situation is precisely the reverse of the case without storage: for $\hat{t} = 0$ we obtain
the smallest possible objective value, namely $1$ (see the right hand side solution in Figure~\ref{fig:flow}).
\begin{figure}[htb]
\begin{minipage}[b]{.45\textwidth}
  \centering
  \begin{tikzpicture}[->,>=stealth',shorten >=1pt,auto,node distance=2.5cm,
   thick,main node/.style={circle,draw,font=\normalfont}]

   \node[main node] (1) {$s$};
   \node[main node] (2) [right of=1] {$v$};
   \node[main node] (3) [right of=2] {$t$};
   \node (1a) [left of=1,node distance =.5cm,anchor=east] {$[0,\,1)$};
   
   \node[main node] (4) [below of = 1,node distance = 1.5cm] {$s$};
   \node[main node] (5) [below of = 2,node distance = 1.5cm] {$v$};
   \node[main node] (6) [below of = 3,node distance = 1.5cm] {$t$};
   \node (4a) [left of=4,node distance =.5cm,anchor=east] {$[1,\,1)$};
   
   \node[main node] (7) [below of = 4,node distance = 1.5cm] {$s$};
   \node[main node] (8) [below of = 5,node distance = 1.5cm] {$v$};
   \node[main node] (9) [below of = 6,node distance = 1.5cm] {$t$};
   \node (7a) [left of=7,node distance =.5cm,anchor=east] {$[1,\,3)$};

   \node[main node] (10) [below of = 7,node distance = 1.5cm] {$s$};
   \node[main node] (11) [below of = 8,node distance = 1.5cm] {$v$};
   \node[main node] (12) [below of = 9,node distance = 1.5cm] {$t$};
   \node (10a) [left of=10,node distance =.5cm,anchor=east] {$[3,\,3)$};

   \path[every node/.style={font=\normalfont}]
     (1) edge node{$2\ (2)$} (2)
     (2) edge node {$0\ (0)$} (3);
   \path[every node/.style={font=\normalfont}]
     (4) edge node{$0\ (0)$} (5)
     (5) edge node {$0\ (0)$} (6);
   \path[every node/.style={font=\normalfont}]
     (7) edge node{$0\ (0)$} (8)
     (8) edge node {$2\ (2)$} (9);
   \path[every node/.style={font=\normalfont}]
     (10) edge node{$0\ (0)$} (11)
     (11) edge node {$0\ (0)$} (12);
   \path[every node/.style={font=\normalfont}]
     (2) edge node{$2\ (2)$} (5)
     (5) edge node {$2\ (2)$} (8)
     (8) edge node {$0\ (2)$} (11);
 \end{tikzpicture}
\end{minipage}\hfill
\begin{minipage}[b]{.45\textwidth}
  \centering
  \begin{tikzpicture}[->,>=stealth',shorten >=1pt,auto,node distance=2.5cm,
   thick,main node/.style={circle,draw,font=\normalfont}]

   \node[main node] (1) {$s$};
   \node[main node] (2) [right of=1] {$v$};
   \node[main node] (3) [right of=2] {$t$};
   \node (1a) [left of=1,node distance =.5cm,anchor=east] {$[0,\,0)$};
   
   \node[main node] (4) [below of = 1,node distance = 1.5cm] {$s$};
   \node[main node] (5) [below of = 2,node distance = 1.5cm] {$v$};
   \node[main node] (6) [below of = 3,node distance = 1.5cm] {$t$};
   \node (4a) [left of=4,node distance =.5cm,anchor=east] {$[0,\,1)$};
   
   \node[main node] (7) [below of = 4,node distance = 1.5cm] {$s$};
   \node[main node] (8) [below of = 5,node distance = 1.5cm] {$v$};
   \node[main node] (9) [below of = 6,node distance = 1.5cm] {$t$};
   \node (7a) [left of=7,node distance =.5cm,anchor=east] {$[1,\,2)$};

   \node[main node] (10) [below of = 7,node distance = 1.5cm] {$s$};
   \node[main node] (11) [below of = 8,node distance = 1.5cm] {$v$};
   \node[main node] (12) [below of = 9,node distance = 1.5cm] {$t$};
   \node (10a) [left of=10,node distance =.5cm,anchor=east] {$[2,\,3)$};

   \path[every node/.style={font=\normalfont}]
     (1) edge node{$0\ (0)$} (2)
     (2) edge node {$0\ (0)$} (3);
   \path[every node/.style={font=\normalfont}]
     (4) edge node{$0\ (0)$} (5)
     (5) edge node {$0\ (0)$} (6);
   \path[every node/.style={font=\normalfont}]
     (7) edge node{$0\ (0)$} (8)
     (8) edge node {$0\ (1)$} (9);
   \path[every node/.style={font=\normalfont}]
     (10) edge node{$1\ (1)$} (11)
     (11) edge node {$1\ (1)$} (12);
   \path[every node/.style={font=\normalfont}]
     (2) edge node{$0\ (2)$} (5)
     (5) edge node {$0\ (2)$} (8)
     (8) edge node {$0\ (2)$} (11);
 \end{tikzpicture}
\end{minipage}
\caption{Optimal flows for $\hat t=1$ (left) and $\hat t=0$ (right) in the case with storage. The arcs are labeled with flow values
  and, in parentheses, capacities, during the time interval indicated at the left of the network.}
  \label{fig:flow}
\end{figure}
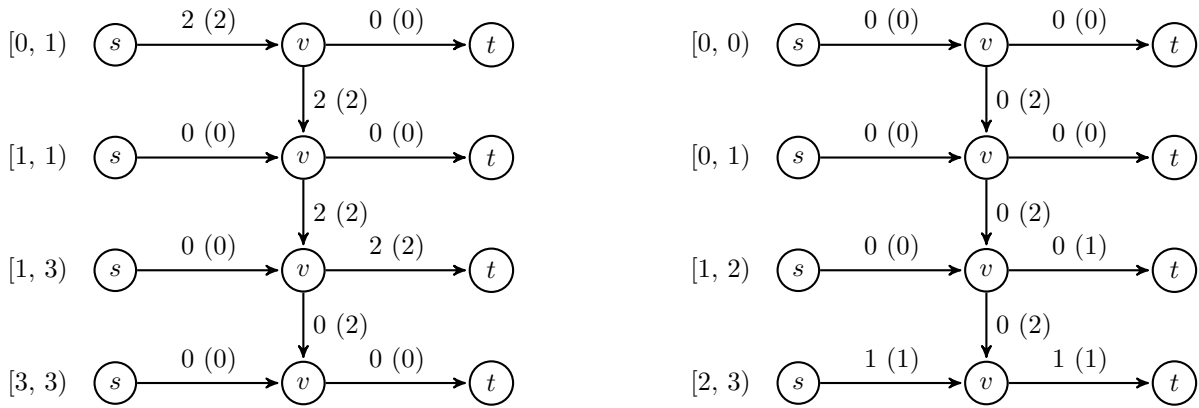

This example shows that ignoring storage can in some sense be as bad as can be: the optimal schedule without storage is exactly the schedule that minimizes flow if storage is allowed, and vice versa.

\subsection{A mixed integer linear programming formulation}\label{subsec:nonlinmodel}
We first model the problem as a nonlinear mixed integer program and then linearise the nonlinear
constraints with the help of additional variables. Since we need at most two time points for each
job to indicate its start and end time, as well as the time horizon start and end times, we use
variables $t_0, t_1,\dots,t_{M-1}, t_M$ where $M=2\lvert A_1\rvert+1$. Clearly we require
\begin{equation}
\label{eq:sorttimes}
  0 = t_0 \leqslant t_1 \leqslant \dots \leqslant t_{M-1} \leqslant t_M = T.
\end{equation}
We introduce binary variable $w_{ai} = 1$ if arc $a\in A_1$ is undergoing maintenance (i.e. is {\em not} available) in time interval $[t_{i-1},t_i)$ and zero otherwise. For convenience in what follows, we include $w_{a0}=0$ for each $a\in A_1$.

If $w_{ai}=1$ then the start time of interval $i$ must be at or after the release date of the job on
$a$. The
implication $w_{ai}=1\implies t_{i-1}\geqslant r_a$ can be modelled linearly with
\begin{equation}
\label{eq:reldate}
   t_{i-1} \geqslant r_{a} w_{ai}, \qquad\text{for all } \ i\in \{1,\ldots,M\},\ a \in A_1.
\end{equation}
Similarly, if $w_{ai}=1$ then the end time of interval $i$ must be at or before the due date of the
job on $a$, i.e., we want to model the implication $w_{ai}=1\implies t_{i}\leqslant d_a$, which can
be done by the linear constraint
\begin{equation}
\label{eq:duedate}
   t_{i} \leqslant d_{a} + (T- d_{a}) (1-w_{ai}), \qquad \forall\ i\in  \{1,\ldots,M\},\ a \in A_1.
\end{equation}
This constraint becomes $t_i\leqslant d_a$ when $w_{ai}=1$, and $t_i\leqslant T$ when $w_{ai}=0$. 

To ensure that the arc is shut precisely for the requisite duration, we constrain the total duration of the time intervals in which the arc is undergoing maintenance to equal the maintenance job processing time. This can be enforced by the nonlinear constraint
\begin{equation}
\label{eq:shuttimeNL}
\sum_{i=1}^M (t_i-t_{i-1}) w_{ai} = p_{a}, \qquad \forall\ a \in A_1.
\end{equation}

To enforce job processing without preemption, we require that each of the vectors $(w_{ai})_{i=0,1,\ldots,M}$ has the consecutive ones property: it consists of a sequence of 0's, then a sequence of 1's, then a sequence of 0's. To enforce this property, we introduce another binary variable: $z_{ai}=1$ if the processing of the job on arc $a \in A_1$ begins at the beginning of interval $i$. The following two constraints ensure that the maintenance job on an arc begins exactly once in the time horizon and once the arc is shut for maintenance then it remains so for consecutive time intervals:
\begin{equation}
\label{eq:consecprop3}
     \sum_{i=1}^{M} z_{ai} = 1 , \qquad  a \in A_1,
\end{equation}
and
\begin{equation}
\label{eq:consecprop1}
   z_{ai} \geqslant w_{ai} - w_{a(i-1)} , \qquad \forall \ i\in  \{1,\ldots,M\},\ a \in A_1.
\end{equation}

To ensure flow cannot pass through the arc while it is shut, we use another nonlinear constraint:
\begin{equation}
\label{eq:flowcapNL}
  x_{ai} \leqslant (t_i-t_{i-1}) (1-w_{ai})u_a, \qquad \forall\ i\in  \{1,\ldots,M\},\ a \in A_1.
\end{equation}

The resulting formulation, which we refer to as the {\em Continuous Time Nonlinear Integer Program}, denoted by CTIP-NL, is given by
\begin{align}
\text{maximize }&\sum_{i=1}^M\left(\sum_{a\in\delta^{\text{out}}(s)}x_{ai}-\sum_{a\in\delta^{\text{in}}(s)}x_{ai}\right)\nonumber\\
\text{s.t.}\quad \sum_{a\in\delta^{\text{out}}(v)}x_{ai} &= \sum_{a\in\delta^{\text{in}}(v)}x_{ai} && i\in \{1,\ldots,M\},\ v\in V\setminus (W\cup\{s,t\}), \nonumber\\
 \sum_{a\in\delta^{\text{out}}(v)}x_{ai}+x_{vi} &= \sum_{a\in\delta^{\text{in}}(v)}x_{ai}+x_{v,i-1} && i\in \{1,\ldots,M\},\ v\in W, \nonumber\\
 x_{ai} &\leqslant (t_i-t_{i-1})u_a && i\in  \{1,\ldots,M\}, a \in A\setminus A_1, \label{eq:flow-int-ub}\\
0 \leqslant x_{vi} &\leqslant u_v && i\in\{0,1,\ldots,M\},\ v\in W, \nonumber\\
x_{v0} &= x_{vM} =0 && v\in W,\nonumber\\
w_{a0} &= 0 && a\in A_1,\nonumber\\
& (\ref{eq:sorttimes}),  (\ref{eq:reldate}), (\ref{eq:duedate}),(\ref{eq:consecprop3}), (\ref{eq:consecprop1}),  (\ref{eq:shuttimeNL}), (\ref{eq:flowcapNL}), && \nonumber\\
x_{ai} &\geqslant 0 && i\in \{1,\ldots,M\},\  a\in A,\nonumber\\
w_{ai} & \in \{0,1\} &&  i\in  \{0,1,\ldots,M\},\ a \in A_1, \ \mbox{and}\nonumber\\
z_{ai} & \in \{0,1\} &&  i\in  \{1,\ldots,M\},\ a \in A_1. \nonumber
\end{align}

The nonlinear constraints in the above formulation can readily be modelled linearly with the use of additional variables. We define $\Delta_{ai}$ and $\bar{\Delta}_{ai}$ for each $a\in A_1$ and $i\in \{1,\ldots,M\}$ by
\begin{align*}
\Delta_{ai} &= \begin{cases} t_i-t_{i-1}, & \text{if } w_{ai} = 1 \\0, & \text{otherwise} \end{cases}\text{ and} & 
\bar{\Delta}_{ai} &= \begin{cases} t_i-t_{i-1}, & \text{if } w_{ai} = 0 \\0, & \text{otherwise.} \end{cases} 
\end{align*}
This can be modelled linearly via the constraints
\begin{equation}
\label{eq:delsum}
   \Delta_{ai} + \bar{\Delta}_{ai} = t_i-t_{i-1}, \qquad \forall\ a\in A_1,\ i\in \{1,\ldots,M\},
\end{equation}
together with
\begin{equation}
\label{eq:dellogic1}
   \Delta_{ai} \leqslant p_{a} w_{ai},  \qquad \forall\ a\in A_1,\ i\in \{1,\ldots,M\},
\end{equation}
and
\begin{equation}
\label{eq:dellogic2}
\bar{\Delta}_{ai} \leqslant  (T- p_{a})(1-w_{ai}) , \qquad \forall\ a\in A_1,\ i\in \{1,\ldots,M\}.
\end{equation}
 Then (\ref{eq:shuttimeNL}) and (\ref{eq:flowcapNL}) can be modelled linearly with
\begin{equation}
\label{eq:shuttime}
  \sum_{i=1}^M {\Delta}_{ai} = p_{a}, \qquad \forall a \in A_1,
\end{equation}
and
\begin{equation}
\label{eq:flowcap}
  x_{ai} \leqslant \bar{\Delta}_{ai}u_a, \qquad \forall i\in \{1,\ldots,M\},\ a \in A_1.
\end{equation}

We refer to the mixed integer linear programming formulation obtained by adding constraints (\ref{eq:delsum}), (\ref{eq:dellogic1}), and (\ref{eq:dellogic2}) to CTIP-NL and replacing the nonlinear constraints  (\ref{eq:shuttimeNL}) and (\ref{eq:flowcapNL})  with (\ref{eq:shuttime}) and (\ref{eq:flowcap}) respectively in CTIP-NL as the {\em Continuous Time Integer Program}, denoted by CTIP.

\section{Properties of optimal solutions}\label{sec:properties}
In this section we derive properties of optimal solutions for both the cases: a) when there are no storage nodes, i.e. $W = \varnothing$, and b) when there are storage nodes, i.e, $W \neq \varnothing$.

For the case when there are no storage nodes, i.e., $W=\varnothing$, we will prove in Lemma~\ref{lem:finiteness} that, without loss of
  generality, we may assume finitely many possible start times for each job. For each arc $a\in A_1$ we will
  construct a finite subset $S(a)\subseteq[r_a,\,d_a-p_a]$, and then prove that there is always an
  optimal solution $\vect{t^*}$ such that $\vect{t^*_a}\in S(a)$ for every $a\in A_1$. Clearly we
  should consider to start the job on arc $a$ as early as possible or as late as possible, so the
  set $S_0(a)=\left\{r_{a},d_{a}-p_{a}\right\}$  should be contained in $S(a)$ for every $a \in
  A_1$. Now suppose that the job on arc $a'$ starts at time $t$. If $t<r_a+p_a$ and $t+p_{a'}<d_a-p_a$, then we
  might consider to start the job on arc $a$ at time $t+p_{a'}$ in order to start the job on arc $a$
  as early as possible while avoiding overlap of the two jobs. By similar reasoning it can be reasonable to start the job
  on arc $a$ at time $t$, at time $t-p_a$ or at time $t+p_{a'}-p_a$ if these are in $[r_a,\,d_a-p_a]$. So we get
  a new candidate start time set $S_1(a)$ for arc $a$ by adding to $S_0(a)$ all times in $[r_a,\,d_a-p_a]$ that can be written in the
  form $t$, $t+p_{a'}$, $t-p_a$, or $t+p_{a'}-p_a$ for some $t\in S_0(a')$. Assuming that we have
  already defined candidate start time sets $S_k(a)$ for some nonnegative integer $k$ and all $a\in
  A_1$, we can extend these sets to sets $S_{k+1}(a)$ in the same way. More formally, for a set $X\subseteq\reals$ and a real number $\lambda$ we write $X+s$ for the set
$\{x+\lambda\ :\ x\in X\}$, and we define recursively,
\begin{multline}\label{eq:start_times_recursion}
S_{k+1}(a)=S_k(a)\cup\bigcup\limits_{a'\in A_1\setminus\{a\}}\Big(S_k(a')\cup\left(S_k(a')+p_{a'}\right)\cup\left(S_k(a')-p_{a}\right)\\
\cup\left(S_k(a')+p_{a'}-p_{a}\right)\Big)\cap [r_{a},d_{a}-p_{a}],
\end{multline}
and finally,
\begin{equation}\label{eq:start_times_union}
S(a)=S_{\lvert A_1\rvert-1}(a).
\end{equation}
We claim that in order to solve~(\ref{eq:master_problem}) it is sufficient to maximize over the finite set $X'=\prod_{a\in A_1} S(a)$. Suppose we have an optimal solution $\vect{t^*}$ and assume $A_0=\left\{a\in A_1\ :\ t^*_{a}\not\in S(a)\right\}\neq\varnothing$. Our argument is based on performing a sequence of modification steps on this solution without losing optimality. In order to describe the single modification step we define a graph on the vertex set $A_1$ associated with the current solution: Two arcs $a,a'\in A_1$ are joined by an edge if $\left\{t^*_{a},\,t^*_{a}+p_{a}\right\}\cap\left\{t^*_{a'},\,t^*_{a'}+p_{a'}\right\}\neq\varnothing$. For $a\in A_1$ denote the connected component of $a$ in this graph by $\mathcal C(a)$, and let $h$ be the distance function for this graph: $h(a,a')$ is the minimal length of a path from $a$ to $a'$ if such a path exists, and $\infty$ otherwise.
\begin{lemma}\label{lem:free_jobs}
For $a\in A_0$ and $a'\in\mathcal C(a)$, $t^*_{a'}\not\in S_{\lvert A_1\rvert-h(a,a')-1}(a')$. In particular, $t^*_{a'}\not\in S_0(a')$ for $a'\in\mathcal C(a)$.
\end{lemma}
\begin{proof}
Let $h=h(a,a')$ and assume $t^*_{a'}\in S_{\lvert A_1\rvert-h-1}(a')$.  Let $a'=a_0,a_1,\ldots,a_h=a$ be a minimal path from $a'$ to $a$. Using~(\ref{eq:start_times_recursion}) we deduce $t^*_{a_i}\in\mathcal S'_{\lvert A_1\rvert-h-1+i}(a')$ for $i=1,2,\ldots,h$. For $i=h$ this is $t^*_{a}\in\mathcal S'_{\lvert A_1\rvert-1}(a)$, contradicting the hypothesis $a\in A_0$.
\end{proof}
\begin{lemma}\label{lem:finiteness}
If $W=\varnothing$, then there is an optimal solution for~(\ref{eq:master_problem}) such that $t^*_{a}\in S(a)$ for each $a\in A_1$.
\end{lemma}
\begin{proof}
By Lemma~\ref{lem:free_jobs}, no job on any arc $a'$ in the component of an arc  $a\in A_0$ starts at one of its boundary start times $r_{a'}$ or $d_{a'}-p_{a'}$. So we can shift all jobs in $\mathcal C(a)$ by $\pm\varepsilon$ for some $\varepsilon>0$ without becoming infeasible and without changing the graph associated with the solution. By optimality, such a shift (in either direction) leaves the objective value unchanged (see Lemma~\ref{lem:shifting} in the appendix). Let $\vect{t^*}(\varepsilon)$ be the solution obtained from $\vect{t^*}$ by shifting the jobs in $\mathcal C(a)$ by $\varepsilon$ to the right, i.e.,
\[t^*(\varepsilon)_{a'}=
\begin{cases}
  t^*_{a'}+\varepsilon & \text{for }a'\in\mathcal C(a),\\
  t^*_{a'} & \text{for }a'\in A_1\setminus\mathcal C(a)).
\end{cases}
\]
Pick the smallest value $\varepsilon$ such that for the solution $\vect{t^*}(\varepsilon)$ we have
\begin{itemize}
\item $\{t^*(\varepsilon)_{a'},\,t^*(\varepsilon)_{a'}+p_{a'}\}\cap\{t^*(\varepsilon)_{a''},\,t^*(\varepsilon)_{a''}+p_{a''}\}\neq\varnothing$ for some $a'\in\mathcal C(a)$ and $a''\in A_1\setminus \mathcal C(a)$, or
\item $t^*(\varepsilon)_{a'}=d_{a'}-p_{a'}$ for some $a'\in\mathcal C(a)$.
\end{itemize}
By construction, the graph corresponding to the new solution $t^*(\varepsilon)$ has more edges than the graph for the original solution, or the size of $A_0$ decreases. Hence, this modification can be iterated only a finite number of times, and this iterated process terminates with an optimal solution with $A_0=\varnothing$.
\end{proof}
By Lemma~\ref{lem:finiteness}, there is an optimal solution in the finite set $X'$, and the proof implies the following integrality property.
\begin{proposition}\label{prop:no_storage}
If the input data is integer and $W=\varnothing$, then there is an integral optimal solution, i.e., an optimal solution $\vect{t^*}$ with $t^*_{a}\in\ints$ for all $a\in A_1$.
\end{proposition}
The statement of Proposition~\ref{prop:no_storage} is in general not true for the problem involving storage, as is indicated by the following example.

\begin{example}\label{ex:nonint}
Consider the network in Figure~\ref{fig:Example_noniteger} over a time horizon $T=7$, and suppose that node $v$ has storage capacity $3$ and we have four jobs with parameters shown in the table in  Figure~\ref{fig:Example_noniteger}.
\begin{figure}[htb]
\begin{minipage}[b]{.48\textwidth}
  \begin{center}
 \begin{tikzpicture}[->,>=stealth',shorten >=1pt,auto,node distance=3.2cm,
   thick,main node/.style={circle,draw,font=\normalfont}]

   \node[main node] (1) {$s$};
   \node[main node] (2) [right of=1] {$v$};
   \node[main node] (3) [ right of=2] {$t$};

   \path[every node/.style={font=\normalfont}]
     (1) edge node{$a\ (4)$} (2)
     (2) edge [bend left=50] node [pos=0.5, above]  {$b\ (2)$} (3)
         edge node {c\ (1)} (3)
         edge[bend right=50] node  [pos=0.5, above ]{$d\ (4)$} (3);
 \end{tikzpicture}
   \caption*{Network}
  \end{center} 
\end{minipage}\hfill
\begin{minipage}[b]{.5\linewidth}
\setlength{\tabcolsep}{10pt}
\centering
\begin{tabular}{@{}ccccc@{}} \toprule
	arc & $r$ & $d$ & $p$ \\ \midrule
    $a$ & $0$ & $5$ & $3$ \\
    $b$ & $3$ & $5$ & $2$ \\
    $c$ & $0$ & $5$ & $5$ \\
    $d$ & $0$ & $6$ & $6$ \\ \bottomrule
  \end{tabular}
\caption*{Job parameters}\label{tab:jobs}
\end{minipage}
\caption{Network and job parameters for Example~\ref{ex:nonint}.}\label{fig:Example_noniteger}
\end{figure}

The job on arc $a$ is the only job that can be moved, and its start time $\hat t$ has to be in the
interval $[0,2]$. 
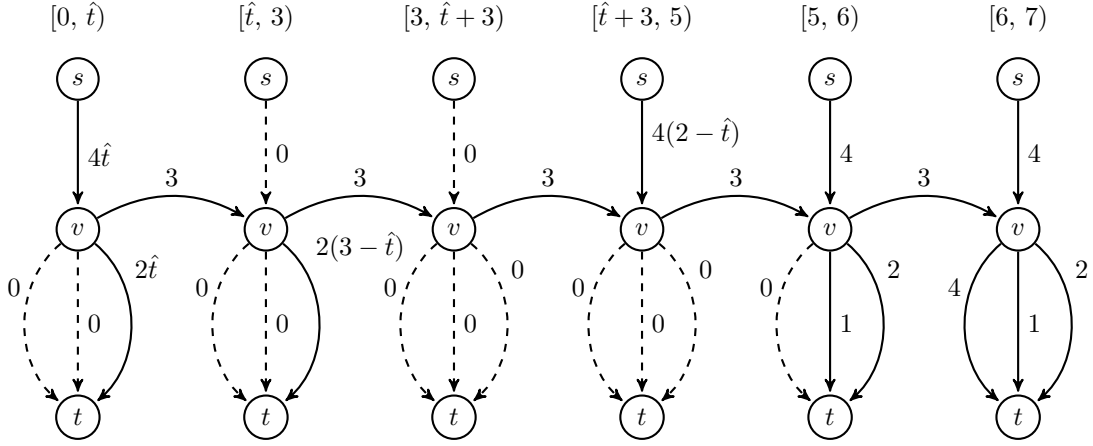
\begin{figure}[htb]
 \centering
 \begin{tikzpicture}[->,>=stealth',shorten >=1pt,auto,node distance=2.5cm,
   thick,main node/.style={circle,draw,font=\normalfont}]

   \node[main node] (1) {$s$};
   \node[main node] (2) [below of=1,node distance =2cm] {$v$};
   \node[main node] (3) [below of=2] {$t$};
   \node (1a) [above of=1,node distance =.5cm,anchor=south] {$[0,\,\hat t)$};
   
   \node[main node] (4) [right of=1] {$s$};
   \node[main node] (5) [right of=2] {$v$};
   \node[main node] (6) [right of=3] {$t$};
   \node (4a) [above of=4,node distance =.5cm,anchor=south] {$[\hat t,\,3)$};
  
   \node[main node] (7) [right of=4] {$s$};
   \node[main node] (8) [right of=5] {$v$};
   \node[main node] (9) [right of=6] {$t$};
   \node (7a) [above of=7,node distance =.5cm,anchor=south] {$[3,\,\hat t+3)$};
  
   \node[main node] (10) [right of=7] {$s$};
   \node[main node] (11) [right of=8] {$v$};
   \node[main node] (12) [right of=9] {$t$};
   \node (10a) [above of=10,node distance =.5cm,anchor=south] {$[\hat t +3,\,5)$};

   \node[main node] (13) [right of=10] {$s$};
   \node[main node] (14) [right of=11] {$v$};
   \node[main node] (15) [right of=12] {$t$};
   \node (13a) [above of=13,node distance =.5cm,anchor=south] {$[5,\,6)$};

   \node[main node] (16) [right of=13] {$s$};
   \node[main node] (17) [right of=14] {$v$};
   \node[main node] (18) [right of=15] {$t$};
   \node (16a) [above of=16,node distance =.5cm,anchor=south] {$[6,\,7)$};
  
   \path[every node/.style={font=\normalfont}]
     (1) edge node{$4\hat t$} (2)
     (2) edge [bend left=50] node [pos=0.3]  {$2\hat t$} (3)
         edge [dashed] node {$0$} (3)
         edge[bend right=50,dashed] node  [pos=0.3,left]{$0$} (3);
   \path[every node/.style={font=\normalfont}]
     (4) edge [dashed] node{$0$} (5)
     (5) edge [bend left=50] node [pos=0.2]  {$2(3-\hat t)$} (6)
         edge [dashed] node {$0$} (6)
         edge[bend right=50,dashed] node  [pos=0.3,left]{$0$} (6);
    \path[every node/.style={font=\normalfont}]
     (7) edge [dashed] node{$0$} (8)
     (8) edge [bend left=50,dashed] node [pos=0.3]  {$0$} (9)
         edge [dashed] node {$0$} (9)
         edge[bend right=50,dashed] node  [pos=0.3,left]{$0$} (9);
   \path[every node/.style={font=\normalfont}]
     (10) edge node[pos=.3] {$4(2-\hat t)$} (11)
     (11) edge [bend left=50,dashed] node [pos=0.3]  {$0$} (12)
         edge [dashed] node {$0$} (12)
         edge[bend right=50,dashed] node  [pos=0.3,left]{$0$} (12);
   \path[every node/.style={font=\normalfont}]
     (13) edge node{$4$} (14)
     (14) edge [bend left=50] node [pos=0.3]  {$2$} (15)
         edge node {$1$} (15)
         edge[bend right=50,dashed] node  [pos=0.3,left]{$0$} (15);
   \path[every node/.style={font=\normalfont}]
     (16) edge node{$4$} (17)
     (17) edge [bend left=50] node [pos=0.3]  {$2$} (18)
         edge node {$1$} (18)
         edge[bend right=50] node  [pos=0.3,left]{$4$} (18);
   \path[every node/.style={font=\normalfont}] 
     (2) edge [bend left=30] node{$3$} (5) 
     (5) edge [bend left=30] node{$3$} (8)
     (8) edge [bend left=30] node{$3$} (11)
     (11) edge [bend left=30] node{$3$} (14)
     (14) edge [bend left=30] node{$3$} (17);
  \end{tikzpicture}
  \caption{The time-expanded networks for the instance in Example~\ref{ex:nonint}. Arcs under
    maintenance are indicated by dashed lines.}
  \label{fig:networks_example}
\end{figure}
The time slicing is given by
\[(t_0,\,t_1,\,\ldots,\,t_6)=(0,\,\hat t,\, 3,\hat t+3,\,5,6,7)\] 
and the corresponding time-expanded network is shown in Figure~\ref{fig:networks_example}
The total capacity of the arcs going into $t$ is  
\[2\hat t+2(3-\hat t)+1+2+4+1+2=16.\]
A total throughput of 16 can be achieved if and only if there is a feasible solution in which all
the arcs into node $t$ are at capacity. In order for arc $b$ to be at capacity in the first two time
periods, it is necessary that the flow on arc $a$ in the first time period is at least 6, which
implies $4\hat t\geqslant 6$, i.e., $\hat t\geqslant 3/2$. On the other hand, the total capacity of
the arcs out of node $s$ is 
\[4\hat t+4(2-\hat t)+4+4=16.\]
So in order to achieve a total throughput of $16$ the arcs out of node $s$ have to be at capacity in
each time period as well. For the first time period, this implies $4\hat t-2\hat t\leqslant 3$, and
therefore $\hat t\leqslant 3/2$. We conclude that for a total throughput of $16$ it is necessary
that the job on arc $a$ starts at time $3/2$. This is also sufficient, as can be seen from the
solution illustrated in Figure~
\begin{figure}[htb]
 \centering
 \begin{tikzpicture}[->,>=stealth',shorten >=1pt,auto,node distance=2.5cm,
   thick,main node/.style={circle,draw,font=\normalfont}]

   \node[main node] (1) {$s$};
   \node[main node] (2) [below of=1,node distance =2cm] {$v$};
   \node[main node] (3) [below of=2] {$t$};
   \node (1a) [above of=1,node distance =.5cm,anchor=south] {$[0,\,3/2)$};
   
   \node[main node] (4) [right of=1] {$s$};
   \node[main node] (5) [right of=2] {$v$};
   \node[main node] (6) [right of=3] {$t$};
   \node (4a) [above of=4,node distance =.5cm,anchor=south] {$[3/2,\,3)$};
  
   \node[main node] (7) [right of=4] {$s$};
   \node[main node] (8) [right of=5] {$v$};
   \node[main node] (9) [right of=6] {$t$};
   \node (7a) [above of=7,node distance =.5cm,anchor=south] {$[3,\,9/2)$};
  
   \node[main node] (10) [right of=7] {$s$};
   \node[main node] (11) [right of=8] {$v$};
   \node[main node] (12) [right of=9] {$t$};
   \node (10a) [above of=10,node distance =.5cm,anchor=south] {$[9/2,\,5)$};

   \node[main node] (13) [right of=10] {$s$};
   \node[main node] (14) [right of=11] {$v$};
   \node[main node] (15) [right of=12] {$t$};
   \node (13a) [above of=13,node distance =.5cm,anchor=south] {$[5,\,6)$};

   \node[main node] (16) [right of=13] {$s$};
   \node[main node] (17) [right of=14] {$v$};
   \node[main node] (18) [right of=15] {$t$};
   \node (16a) [above of=16,node distance =.5cm,anchor=south] {$[6,\,7)$};
  
   \path[every node/.style={font=\normalfont}]
     (1) edge node{$6$} (2)
     (2) edge [bend left=50] node [pos=0.2]  {$3$} (3)
         edge [dashed] node {$0$} (3)
         edge[bend right=50,dashed] node  [pos=0.2,left]{$0$} (3);
   \path[every node/.style={font=\normalfont}]
     (4) edge [dashed] node{$0$} (5)
     (5) edge [bend left=50] node [pos=0.2]  {$3$} (6)
         edge [dashed] node {$0$} (6)
         edge[bend right=50,dashed] node  [pos=0.2,left]{$0$} (6);
    \path[every node/.style={font=\normalfont}]
     (7) edge [dashed] node{$0$} (8)
     (8) edge [bend left=50,dashed] node [pos=0.2]  {$0$} (9)
         edge [dashed] node {$0$} (9)
         edge[bend right=50] node  [pos=0.2,left]{$0$} (9);
   \path[every node/.style={font=\normalfont}]
     (10) edge node[pos=.3] {$2$} (11)
     (11) edge [bend left=50,dashed] node [pos=0.2]  {$0$} (12)
         edge [dashed] node {$0$} (12)
         edge[bend right=50,dashed] node  [pos=0.2,left]{$0$} (12);
   \path[every node/.style={font=\normalfont}]
     (13) edge node{$4$} (14)
     (14) edge [bend left=50] node [pos=0.2]  {$2$} (15)
         edge node {$1$} (15)
         edge[bend right=50,dashed] node  [pos=0.2,left]{$0$} (15);
   \path[every node/.style={font=\normalfont}]
     (16) edge node{$4$} (17)
     (17) edge [bend left=50] node [pos=0.2]  {$2$} (18)
         edge node {$1$} (18)
         edge[bend right=50] node  [pos=0.2,left]{$4$} (18);
   \path[every node/.style={font=\normalfont}] 
     (2) edge [bend left=30] node{$3\ (3)$} (5) 
     (5) edge [bend left=30] node{$0\ (3)$} (8)
     (8) edge [bend left=30] node{$0\ (3)$} (11)
     (11) edge [bend left=30] node{$2\ (3)$} (14)
     (14) edge [bend left=30] node{$3\ (3)$} (17);
  \end{tikzpicture}
  \caption{The flow of value $16$ for $\hat t=3/2$. All arcs except the arcs between copies of the storage node
    $v$ are at capacity.}
  \label{fig:example_solution}
\end{figure}
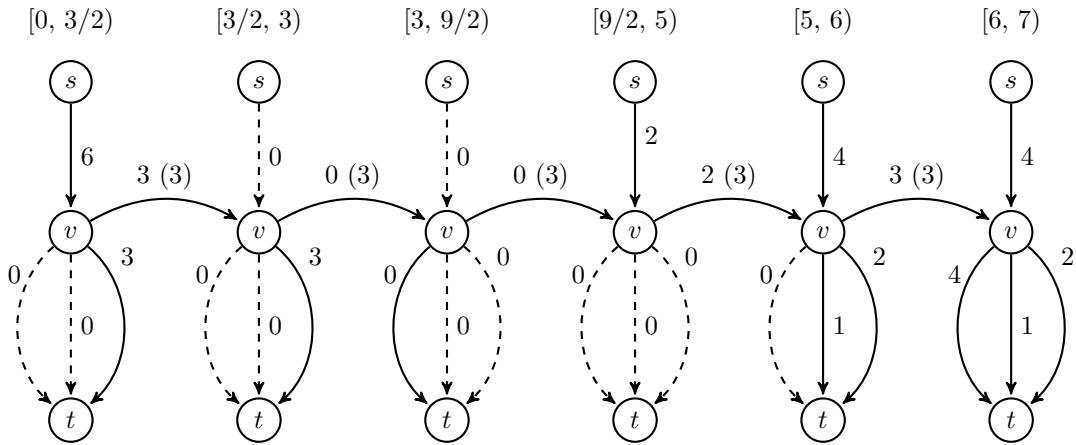
\end{example}

In the following proposition we prove that if there are storage nodes, i.e $W \neq \varnothing $, then we can restrict our attention to rational start times for jobs.
\begin{proposition}\label{prop:rational_start_time}
If $W \neq \varnothing $, then there exists an optimal solution in which all jobs  start at rational times. Further, the denominators of these optimal rational start times do not depend on parameters of the jobs (processing times, release dates or deadlines).
\end{proposition}
\begin{proof}
Observe that the constraint matrix and the right hand side vector for the CTIP formulation contain only integer entries. So by the fundamental theorem for Mixed Integer Linear Programs~\cite{Meyer1974} the convex hull of the set of feasible solutions for CTIP is a rational polyhedron, with all its extreme points rational. Hence there exists an optimal solution with rational start times.

Now this optimal solution of CTIP is obtained at some leaf node of the corresponding branch and bound tree. At this leaf node, all binary variables in CTIP are fixed, and the start times of the jobs are obtained by solving the resulting linear programming formulation so as to yield an extreme point (basic feasible) solution. Such an extreme point has the form $B^{-1}b$, where $B$ is a square nonsingular submatrix of the constraint matrix and $b$ is a corresponding vector of the constant terms in the constraints. By Cramer's rule, $B^{-1} = \frac{1}{|B|}B^*$ where $B^*$ and $|B|$ are the adjoint and determinant of the matrix $B$ respectively. For integer data, the entries of $B^*$ and $b$ are integer, as is $|B|$. So the (rational) solution to the linear program consists of integer multiples of $\frac{1}{|B|}$. Now the claim follows by observing that once all binary variables are fixed in CTIP, the entries in the constraint matrix of the remaining linear program are $0$, $\pm 1$, or an element of $\{\pm u_a : a\in A\}$, and all entries of the right hand side vector are integers. In particular, the only constraints with coefficients {\em not} in $\{-1,0,1\}$ are (\ref{eq:flow-int-ub}) and (\ref{eq:flowcap}).
\end{proof}

Although the exact continuous time formulation CTIP has allowed us to establish the above important property of the problem with storage, it unfortunately performs very poorly in practice (as we shall see in Section~\ref{sec:computations}). Thus other, more efficient, approaches to solving the problem are of interest. In the next section we consider such approaches to finding upper bounds, and in the subsequent section consider lower bounds.

\section{Upper bounds}\label{sec:bounds}
In this section we present an approximate MIP model which gives upper bounds for the problem. This upper bound model can be obtained from any given discretization $0=t_0<t_1<\cdots<t_n=T$ of the time horizon. Importantly, here the $(t_i)_{i=0,\dots,n}$ are fixed input parameters, not decision variables as they were in the CTIP model, and $n$ is an arbitrary (given) positive integer. For $i\in\{1,\ldots,n\}$, we refer to $[t_{i-1},t_i)$ as interval $i$. For $ a \in A_1$, let $\mathcal S_a$ be the set of intervals in which job $a$ can start, and let $\mathcal T_a$ be the set of intervals in which job $a$ can be partially processed. More precisely, $\mathcal S_a=\{i\ : r_{a}<t_{i},\ d_{a}-p_{a}\geqslant t_{i-1}\}$ and $\mathcal T_{a}=\{i\ :\ d_{a}>t_{i-1},\ r_{a}<t_{i}\}$. Our upper bound model has the following variables:
\begin{itemize}
\item $x_{ai}$: total flow on arc $a$ in interval $i$ ($a\in A$, $i\in\{1,\ldots,n\}$),
\item $x_{vi}$: amount of flow stored in node $v$ at time $t_i$ ($v\in W$, $i\in\{1,\ldots,n\}$),
\item $y_{ai}$: binary start indicator, i.e., $y_{ai}=1\iff t_{i-1}\leqslant t^*_{a}<t_{i}$ ($ a \in A_1$, $i\in\mathcal S_a$), and
\item $z_{ai}$: fraction of interval $i$ for which job $a$ is processed ($ a \in A_1$, $i\in\mathcal T_a$).
\end{itemize}
To ensure that every job is processed exactly once, we have the constraints
\begin{align}
\sum_{i\in \mathcal S_a}y_{ai} &= 1 && \forall\ a \in A_1. \label{eq:job_starts}
\end{align}
Clearly, the job on an arc $a \in A_1$ cannot be processed for more than  $\min\{t_i,d_a\} - \max\{t_{i-1},r_a\}$ time in an interval $i \in \mathcal T_a$. So we add the constraints
\begin{align}
  (t_i-t_{i-1})z_{ai} &\leqslant \min\{t_i,d_a\} - \max\{t_{i-1},r_a\} && \forall\  a \in A_1, i\in \mathcal T_a. \label{eq:bounds_z}
\end{align}
The processing times are enforced by the constraints
\begin{align}
  \sum_{i\in\mathcal T_a}(t_i-t_{i-1})z_{ai} &= p_{a} && \forall\ a \in A_1. \label{eq:processing_TDIP}
\end{align}
Next we describe how the $z$ variables and the $y$ variables are linked. For $a \in A_1$ and $i\in\mathcal S_a$, let $Q_{ai} \subseteq \mathcal T_a $ be the set of intervals that can be affected by job $a$ when it starts in interval $i$, i.e.,
\[Q_{ai}=\{k \in \mathcal T_a\ :\ k\geqslant i,\ t_{i}+p_{a}>t_{k-1}\}.\]
When job $a$ starts in interval $i$, then it has to be completed within $Q_{ai}$, which gives the constraints
\begin{align}
  \sum_{k\in Q_{ai}}(t_i-t_{i-1})z_{ak} &\geqslant p_{a}y_{ai} && \forall\ a \in A_1,\ i\in\mathcal S_a. \label{eq:processing_2_TDIP}
\end{align}
In the opposite direction, let $P_{ai}$ be the set of intervals $k$ such that starting job $a$ in interval $k$ can affect interval $i$.  Let $P^*_{ai}$ be the subset of these intervals $k$ such that starting job $a$ in interval $k$ closes arc $a$ for the whole interval $i$ or until completion of job $a$. More precisely,
\begin{align*}
P_{ai} &= \{k \in \mathcal S_a\ : t_{k-1}\leqslant t_{i-1}<t_{k}+p_{a}\}, \\
P^*_{ai} &= \{k \in P_{ai}\ : \ \max\{t_{k-1},r_{a}\}+p_{a}\geqslant\min\{t_{i},d_{a}\}\}.
\end{align*}
For $k\in P_{ai}$, let $\mu^+_{aki}$ and $\mu^-_{aki}$ be upper and lower bounds for $(t_i-t_{i-1})z_{ai}$ if job $a$ starts in interval $k$, i.e.,
\begin{align*}
\mu^+_{aki} &= \min\{t_{i},d_{a}, t_k+p_{a}\}- \max\{t_{i-1},\,r_{a}\}\\
\mu^-_{aki} &=
\begin{cases}
\max\{0,\,\min\{t_k,\,d_{a}\} - (d_{a}-p_{a})\} & \text{for }k = i ,\ i \in \mathcal S_a, \\
\min\{t_{i},\,d_{a}\}-\max\{t_{i-1},\,r_{a}\} & \text{for }k\in P^*_{ai}\setminus \{i\}, \\
\max\{0,\,\max\{t_{k-1},\,r_{a}\}+p_{a}-t_{i-1}\} & \text{for }k\in P_{ai}\setminus (P^*_{ai} \cup \{i\}).
\end{cases}
\end{align*}
We add the constraints
\begin{equation}
\sum_{k\in P_{ai}} \mu^-_{aki}y_{ak} \ \ \leqslant\ \  (t_i-t_{i-1})z_{ai}\ \ \leqslant \sum_{k\in P_{ai}} \mu^+_{aki}y_{ak} \qquad\qquad  \forall\ a \in A_1,\ i\in \mathcal T_a.\label{eq:z_bounds_TDIP}
\end{equation}
 The arc capacities are described by
\begin{align}
x_{ai} &\leqslant (t_{i}-t_{i-1})(1-z_{ai})u_a && \forall\ a\in A_1,\ i\in\mathcal T_a, \label{arc_cap_2_TDIP_1}\\
x_{ai} &\leqslant (t_{i}-t_{i-1})u_a && \forall\ a\in A_1,\ i\in\{1,\ldots,n\}\setminus \mathcal T_a.\label{arc_cap_2_TDIP}
\end{align}
Finally adding the flow conservation constraints~(\ref{eq:flowcon_1_orig}) and~(\ref{eq:flowcon_2_orig}), the arc capacity constraints~(\ref{eq:arc_cap_orig1}) for the arcs in the set $A\setminus A_1$, together with the storage node capacity constraints~(\ref{eq:node_cap_orig}) and~(\ref{eq:boundary_orig}), the upper bound for~(\ref{eq:master_problem}) associated with the given time discretization is
\begin{equation}\label{eq:TDIP}
\varphi=\max_{(\vect x, \vect y, \vect z)\in F} \Bigg\{\sum_{i=1}^{n}\left(\sum_{a\in\delta^{\text{out}}(s)}x_{ai}-\sum_{a\in\delta^{\text{in}}(s)}x_{ai}\right)\ :\ \text{(\ref{eq:flowcon_1_orig}), (\ref{eq:flowcon_2_orig}), (\ref{eq:arc_cap_orig1}), (\ref{eq:node_cap_orig}), (\ref{eq:boundary_orig}), (\ref{eq:job_starts})--(\ref{arc_cap_2_TDIP})}\Bigg\}
\end{equation}
where
$$
  F = \reals_{\geqslant 0}^{(\lvert A\rvert+\lvert W\rvert)n} \times \{0,1\}^{\sum_{ a \in A_1}\lvert \mathcal S_a\rvert}\times [0,1]^{\sum_{ a \in A_1}\lvert\mathcal T_a\rvert}.
$$
We refer to this formulation as the {\em Time Discretized Integer Program}, denoted by TDIP. Since it consists only of constraints that must be satisfied by any feasible solution, we have the following result.
\begin{proposition}\label{prop:upper_bound}
Let $0=t_0<t_1<\cdots<t_n=T$ be any time discretization. Then $\varphi^*\leqslant \varphi$, where $\varphi^*$ is the optimal value of the original problem~(\ref{eq:master_problem}) and $\varphi$ is the optimal value of~(\ref{eq:TDIP}).
\end{proposition}
In Section~\ref{sec:computations}, we experiment with two variants of TDIP. The time indexed model TDIP(TI) refers to unit time discretization, i.e., $n=T$ and $t_i=i$ for $i=0,1,\ldots,T$, while the release date/deadline model TDIP(RD) refers to the discretization consisting of the release times and deadlines of all jobs, i.e. the time discretization with $\{t_i\ :\ i=0,1\dots,n\}$ equal to the set
$$
 \mathcal D := \{r_a\ :\ a\in A_1\}\cup \{d_a\ :\ a\in A_1\}\cup\{0,T\}.
$$

\section{Lower bounds}\label{sec:lower_bounds}
If the same model employs a discretization $0=t_0 < t_1 < \dots < t_n = T$ with the property that
\begin{enumerate}
\item[(i)] $\bigcup_{a\in A_1} \{r_{a},d_{a}\} \subseteq \Gamma \stackrel{\text{def}}{=} \{t_0,t_1,\dots,t_n\}$, and
\item[(ii)] for any $i\in \{0,1,\dots,n\}$ and any $a\in A_1$, if $t_i\in [r_{a},d_{a}-p_{a}]$ then $t_i + p_{a} \in \Gamma$ and if $t_i\in [r_{a}+p_{a},d_{a}]$ then $t_i - p_{a} \in \Gamma$,
\end{enumerate}
then with minor modifications to some model parameters and the additional restriction that $z_{ai}\in \{0,1\}$ for all $a\in A_1$, $i\in\mathcal T_j$, any feasible solution to the model will provide a lower bound on the optimal value of the original problem~(\ref{eq:master_problem}). We refer to a discretization satisfying (i) and (ii) as {\em conformal}, meaning conforming with the job parameters. To derive the modifications, we first re-interpret the variables:
\begin{itemize}
\item $y_{ai}$: binary start indicator, i.e., $y_{ai}=1\iff t_{i-1} = t^*_{a}$ ($a\in A_1$, $i\in\mathcal S_a$), and
\item $z_{ai}$: binary indicator that job $a$ is processed for all of interval $i$ ($a\in A_1$, $i\in\mathcal T_a$).
\end{itemize}
Note that for a conformal discretization, any job starting at the start of a time interval must end at the end of a time interval. We now re-define
\[P_{ai} = \{ k\in {\cal S}_a\ :\ t_{k-1} + p_{a} \geqslant t_i\},\]
which for a conformal discretization yields the set of intervals $k$ for which if the job starts at $t_{k-1}$ it is processed for the whole duration of interval $i$. Also observe that in this case, for $k \in P_{ai}$, $\mu^+_{aki} = t_i - t_{i-1}$ and we re-define $\mu^-_{aki} = t_i - t_{i-1}$ also. As a consequence, (\ref{eq:z_bounds_TDIP}) simplifies to
\begin{align}
  z_{ji} &= \sum_{k\in P{ai}} y_{jk} && \forall\ a\in A_1,\ i\in \mathcal T_a.\label{eq:z_bounds_TDIPLB}
\end{align}
The complete model is given by
\begin{multline}\label{eq:TDIPLB}
\underline{\varphi}=\max_{(\vect x, \vect y, \vect z)\in F} \Bigg\{\sum_{i=1}^{n}\left(\sum_{a\in\delta^{\text{out}}(s)}x_{ai}-\sum_{a\in\delta^{\text{in}}(s)}x_{ai}\right)\ :\ \text{(\ref{eq:flowcon_1_orig}), (\ref{eq:flowcon_2_orig}), (\ref{eq:arc_cap_orig1}), (\ref{eq:node_cap_orig}), (\ref{eq:boundary_orig})},\\
\text{(\ref{eq:job_starts})--(\ref{eq:processing_2_TDIP}), (\ref{arc_cap_2_TDIP_1}),  (\ref{arc_cap_2_TDIP}), (\ref{eq:z_bounds_TDIPLB})},
z \in \{0,1\}^{\sum_{j\in J}\lvert\mathcal T_j\rvert}.\Bigg\}
\end{multline}
We have the following proposition.
\begin{proposition}\label{prop:lower_bound}
Let $0=t_0<t_1<\cdots<t_n=T$ be a conformal time discretization. Then $\varphi^*\geqslant \underline{\varphi}$, where $\varphi^*$ is the optimal value of the original problem~(\ref{eq:master_problem}), and $\underline{\varphi}$ is the optimal value of~(\ref{eq:TDIPLB}).
\end{proposition}
As a consequence, if the discretization $\Gamma$ contains $\{t^*_{a} : a \in A_1\}$, where $\vect t^*$ is an optimal solution of~(\ref{eq:master_problem}), it must be that $\underline\varphi = \val(\vect t^*)$.

Note that for integer data, the unit time discretization is conformal. Thus TDIP(TI) with the additional restriction $z_{ai}\in\{0,1\}$ for all $a\in A_1$, $i\in\mathcal T_a$, yields a lower bound (and feasible solution) for~(\ref{eq:master_problem}). This model is also attractive since by Proposition~\ref{prop:no_storage} it yields an optimal solution for~(\ref{eq:master_problem}) in the case that there are no storage nodes. We denote this final variant with binary $z$ variables by TDIP(TI)$^\text{B}$.

The magnitude of $n$ in any conformal discretization is likely to be large, and hence the lower
bound integer programming model is likely to be slow to solve. Thus we also seek alternative methods
of generating feasible solutions (and hence lower bounds) efficiently. 
For each of the formulations presented so far, the values of the $z$ variables obtained by solving LP
relaxations, for instance in a Branch\&Bound tree, may contain information that can be useful in
guiding construction heuristics for finding good schedules. Although such $z$ values may be fractional, and may have positive values that, for a given job, are not consecutive over time, their ``spread'' across intervals and ``intensity'' within intervals may nevertheless be a useful guide for simple ``repair'' heuristics. We make these ideas precise in what follows, describing two such heuristics. Both heuristics take as input a vector $z$ with $z_{ai} \in [0,1]$ for all $a\in A_1$ and $i\in {\cal T}_a$. Both derive start times $t^*_{a}$ for all $a\in A_1$, which can then be evaluated by solving the max flow problem~(\ref{eq:objective_orig})--(\ref{eq:domains_orig}), so providing a lower bound for~(\ref{eq:master_problem}).

In our computational results, we apply these heuristics by solving a TDIP formulation with a given (relatively short) time limit; we then use the $z$ vector from the LP solution at the node of the branch and bound tree yielding the best upper bound, and the $z$ vector from the best feasible solution for the formulation found so far, as input to the heuristics, in turn.

\subsection{The Projection heuristic}\label{subsec:projection}

In this heuristic we use the ``intensity'' of the $(z_a)_{i=1,\dots,n}$ variables for each $a\in A_1$ as a guide to the choice of start time for the job, $t_a^*$. A given start time $t^*_a$ can be thought of as inducing a vector $(\xi_a)_{i=1,\dots,n}$ by interpreting $\xi_{ai}$ as the proportion of interval $i$ in which the job is being processed if the job starts at time $t^*_a$ and is processed continuously until time $t^*_a+p_a$. This heuristic chooses the vector of start times $t^*$ so as to minimize the $\ell_1$-norm distance between the $\xi$ values induced by $t^*$ and the given $z$. In this sense, the heuristic ``projects'' $z$ onto the set of feasible schedules. Details of the method are as follows.

For $t\in[r_{a},d_{a}-p_{a}]$ and $i\in\mathcal T_a$, let $l_{ai}(t)$ be the time for which job $a$ is processed in interval $i$ when it starts at time $t$, i.e.,
\[l_{ai}(t)=
\begin{cases}
\min\{t_{i},\,t+p_{a}\}-\max\{t,\,t_{i-1}\} & \text{if }t<t_i\text{ and }t+p_{a}>t_{i-1},\\
0 &\text{otherwise}.
\end{cases}\]
The total deviation between the actual processing times in the intervals and the given values $(z_{a})_{i=1,\dots,n}$ is measured by $f(t)=\sum_{i\in\mathcal T_a}\left\lvert(t_i - t_{i-1})z_{ai} - l_{ai}(t)\right\rvert$, and we choose the start time $t^*_a$ as
\begin{equation}\label{eq:proj}
t^*_{a}\in\argmin\left\{f(t)\ :\ t\in[r_{a},d_{a}-p_{a}]\right\}.
\end{equation}
The value $t^*_a$ can be determined as follows. For $i\in \mathcal S_a$, let $\mathcal E_{ai}=\{k\ :\ t_{i-1}+p_{a}\leqslant t_k,\ t_i+p_{a}>t_{k-1}\}$ be the set of intervals in which job $a$ can be completed when it starts in interval $i$. For $k\in\mathcal E_{ai}$, a minimizer of $f(t)$ such that job $a$ starts in interval $i$ and ends in interval $k$ is $t^*_{aik}=t_i-\alpha$, where $(\alpha,\beta)$ is a minimizer of $\lvert\alpha-(t_i - t_{i-1})z_{ai}\rvert+\lvert\beta-(t_k - t_{k-1})z_{ak}\rvert$ subject to $\alpha+\beta=p_{a}-(t_{k-1}-t_i)$ and $\alpha,\beta\geqslant 0$, and finally we can put $t^*_{a}=\argmin\{f(t^*_{aik})\ :\ i\in\mathcal S_a,\ k\in\mathcal E_{ai}\}$.

\subsection{The Centre-of-Mass heuristic}\label{subsec:COM}

In this heuristic we use the ``spread'' of the $(z_a)_{i=1,\dots,n}$ variables for each $a\in A_1$ as a guide to the choice of start time for the job, $t_a^*$. For each arc $ a \in A_1 $ we view the shutdown times $\{(t_i-t_{i-1})z_{ai}\ :\ i\in \mathcal T_a\}$ as a distribution of the ``mass'' of its job processing time over the time horizon $ T $. For $t$ the unique point where this distribution of mass is balanced, (half is distributed earlier than $t$ and half later), we schedule the job for arc $a$ at time $t^*_a = t-p_a/2$. Details of the method are as follows.

Consider an arc in $A_1$, let $i=\min\{i'\ :\ i'\in \mathcal T_a\}$ be the first interval that can be affected by job $a$, and let
\[h=\min\left\{i'\in\mathcal T_a\ :\ \sum_{k=i}^{i'}(t_k - t_{k-1})z_{ak}>\frac{p_a}{2}\right\}.\]
We define the \emph{midpoint} of job $a$ (with respect to $\vect z$) to be the point $t_{\text{mid}}\in[t_{h-1},t_h]$ such that
\[\sum_{k=i}^{h-1}(t_k - t_{k-1})z_{ak}+(t_{\text{mid}}-t_{h-1})z_{ah}=\frac{p_{a}}{2},\]
and then we determine the start time $t^*_{a}$ such that this midpoint exactly halves the processing period of job $a$, i.e., $t^*_{a}=t_{\text{mid}}-p_{a}/2$.

\section{Computational experiments}\label{sec:computations}

In this section we computationally evaluate, on a large set of instances, the performance of the exact formulation CTIP, the upper bounding formulations TDIP(RD) and TDIP(TI), and the lower bounding formulation TDIP(TI)$^B$, together with the two LP based heuristics: the Centre of Mass (CoM) heuristic and the Projection Heuristic (Proj). We first describe the  test data set,  followed by the description of the performance measures used, before presenting the computational results.

\subsection{Test instances}
For our computational study, we use a subset of the randomly generated test instances from~\cite{boland2012scheduling}. 
For each network, we consider the ten instances where all jobs have a time window in the range $[25,35]$ (the second, harder, instance set in~\cite{boland2012scheduling}). These instances all have a time horizon of $T=1000$. Some instance parameters and upper bound MIP dimensions are given in Table~\ref{tab:instances}. 
We introduce storage in these instances by selecting one ``central'' node in each network to be a storage node. We tested a wide range of storage node capacities for that node. These tests showed that, for smaller networks, for values higher than about $c_2=20$, ($c_2$ is a key parameter for arc capacities in the random generator), we found that the instances became very easy: 
in relatively short computing time (within a couple of minutes) the TDIP(TI)$^B$ model could be solved to give a feasible solution with the same value as the TDIP(RD) LP relaxation. Here we report results for all instances with storage capacities of $5$, $10$, $15$ and $20$, making a total of $8 \times 10\times 4 = 320$ instances tested.

We also investigate the performance of the formulations and heuristics on the two instances derived from a problem arising in the scheduling of maintenance for a coal supply chain, the Hunter Valley Coal  Chain (HVCC), studied in~\cite{boland2012scheduling}. The HVCC is the world's largest coal export operation, handling coal mined in the Hunter Valley region of New South Wales, Australia. A record 150.5 million tonnes of coal was exported by the HVCC to customers around the world in the year 2013, with more than 1400 coal vessels per year served by the Port of Newcastle, located at the mouth of the Hunter River. There are three coal handling terminals at the port, for the storage of coal prior to loading onto vessels. Coal from around 35 coal mines is transferred to the coal handling terminals via a network of rail tracks spread over 450 km, using approximately 22,000 train trips per year.  Each terminal has its storage space divided into stockpads on which coal stockpiles are assembled, with four stockpads each in Terminals 1 and 2, and with Terminal 3 modelled as a single stockpad. Coal is modelled as flowing through a network representing the combination of the rail network and the network of handling and conveying equipment at each terminal, with each stockpad modelled as a node with storage. The storage capacities for five of the stockpads is approximately 150 kilotonnes and for the remaining four is approximately 700 kilotonnes.

The components of this supply chain, such as rail tracks, conveyor belts, and stacking, reclaiming and shiploading machinery, all require regular preventive maintenance. However each maintenance plan has some negative impact on the throughput of network, as a component becomes unavailable for use during maintenance. Thus to meet increasing demand for coal, it is crucial to find maintenance schedules that minimize the impact of maintenance on the throughput of the system. MIP-based heuristics that consider only a sparse subset of possible maintenance job start times are used in~\cite{boland2012mixed} to obtain practical schedules, tested on 2010 and 2011 HVCC annual maintenance schedules. Seeking to combat the challenge posed by the very large number of possible job start times, matheuristics are developed and tested on two instances of the problem we consider here, derived from the same 2010 and 2011 HVCC schedules (\cite{boland2012scheduling}).  However these matheuristics exploit the decomposable structure of the network flow problem that occurs if storage is ignored: the methods in~\cite{boland2012scheduling} cannot be applied to the problem with storage and the instances tested in~\cite{boland2012scheduling} disallow storage at the stockpad nodes. As demonstrated in Section~\ref{subsec:example}, the optimal schedule for an instance of the problem obtained by ignoring storage can be very far from the optimal schedule for the instance when storage is considered. Here we analyze the performance of the formulations and heuristics described in this paper on the two instances used in~\cite{boland2012scheduling}, but considering storage at the stockpads, with their capacities set to the values given above.  The time horizon for both instances is  T = 365$\times$24 = 8760 (with a discretization of 1 hour for one year). The 2010 and 2011 instances respectively contain 1457 and 1234 jobs. Every job has a time window
of two weeks and a processing time between an hour and several days.
The instance and upper bound MIP dimensions for these two instances are given in Table~\ref{tab:HVCCC_instances}.

\begin{table}
     \centering
  \scalebox{0.7}{
\renewcommand{\arraystretch}{1.2}
    \begin{tabular}{ccS[table-format=3.0]S[table-format=4.1]S[table-format=3.1]S[table-format=6.1]S[table-format=6.1]S[table-format=7.1]S[table-format=5.1]S[table-format=6.1]S[table-format=6.1]S[table-format=7.1]S[table-format=5.1]}
    \toprule
    Network & \multicolumn{4}{c}{Dimensions} & \multicolumn{4}{c}{TDIP(RD) Size} & \multicolumn{4}{c}{TDIP(TI) Size} \\
    \cmidrule( r){1-1}\cmidrule(lr){2-5}\cmidrule(lr){6-9}\cmidrule(l ){10-13}
          & Nodes & {Arcs}  & {Jobs}  & {$\lvert \mathcal D \rvert$} & {\# Rows} & {\# Columns} & {\# Nonzeros}
          & {\# Binaries} & {\# Rows} & {\# Columns} & {\# Nonzeros} & {\# Binaries} \\ 
\cmidrule(lr){2-5}\cmidrule(lr){6-9}\cmidrule(l ){10-13}
    1     & 12    & 32    & 303.2 	& 456.3 & 25943.4 	& 22871.5 	 	&155200.2 	&5883.6 	&62288.0 & 51149.5 & 682573.4 & 9241.9 \\
    2      & 16    &44    & 421.0   	& 568.7 & 43740.3 	& 39859.6 		& 287528.7 	& 11297.9 	&87449.3 & 71838.5 & 951618.6 & 12824.5 \\
    3     & 18    & 57    & 542.4 	& 658.8 & 63451.8 	& 60600.3 	 	&444662.4 	& 18709.0 	&112601.1 & 94005.9 & 1239069.5 & 16528.1 \\
    4     & 27    & 90    & 847.5 	& 812.2 & 115272.9 	& 117067.3 		& 849284.5 	& 42105.8 	&176571.2 & 148011.7 & 1941644.0 & 26001.1 \\
    5     & 36    & 123   & 1155.9 	& 897.5 & 214480.0 	& 180212.0 	 	&2172605.0 	& 31506.4 	&239884.6 & 201350.0 & 2647223.0 & 35258.8 \\
    6     & 32    & 92    & 873.8 	& 818.3 & 149522.0 	& 122561.0 	 	&1389731.0 	& 21681.1 	&183799.3 & 150383.2 & 1993952.6 & 26643.1 \\
    7     & 48    & 176   & 1657.0  & 963.7 & 327382.0 	& 275677.0 		& 3524593.0 & 48517.6 	&340652.3 & 286671.7 & 3782451.1 & 50491.7 \\
    8     & 64    & 240   & 2268.2 	& 987.5 & 459361.0 	& 387151.0 	 	&5064772.0 	& 68177.5 	&465869.9 & 392573.7 & 5194776.4 & 69172.4 \\
    \bottomrule
    \end{tabular}}
\caption{Sizes of the random networks and average(arithmetic mean) problem sizes for TDIP(RD) and TDIP(TI) after presolve. Each row summarizes statistics for ten instances, each with identical network and arc capacities, but different, randomly generated, jobs.}\label{tab:instances}
\end{table}

\begin{table}
  \centering
  \scalebox{0.74}{
\renewcommand{\arraystretch}{1.2}
    \begin{tabular}{ccccccccccccc}
    \toprule
    Year  & \multicolumn{4}{c}{Dimensions} & \multicolumn{4}{c}{TDIP(RD) Size} & \multicolumn{4}{c}{TDIP(TI) Size} \\
    \cmidrule( r){1-1}\cmidrule(lr){2-5}\cmidrule(lr){6-9}\cmidrule(l ){10-13}
          & Nodes & Arcs  & Jobs  & $\lvert \mathcal D \rvert$ & \# Rows & \# Columns & \# Nonzeros
          & \# Binaries & \# Rows & \# Columns & \# Nonzeros & \# Binaries \\ \cmidrule(lr){2-5}\cmidrule(lr){6-9}\cmidrule(l ){10-13}
    2010  & 109   & 176   & 1458  & 1990  & 293208 & 278451 & 2040868 & 54105 & 1125678 & 1554161 & 10429886 & 502577 \\
    2011  & 109   & 176   & 1235  & 1709  & 251706 & 235847 & 1653790 & 46284 & 1108575 & 1521901 & 10099289 & 493102 \\
    \bottomrule
    \end{tabular}}
 \caption{Sizes of the HVCCC network and average problem sizes for TDIP(RD) and TDIP(TI) after presolve.}\label{tab:HVCCC_instances}
\end{table}

\subsection{Performance Measures}
To evaluate the quality of an upper bound for an instance $I$, we use the percentage gap between the upper bound and the best known lower bound for $I$, i.e. $ (\varphi_I - \varsigma ^*_I)/\varsigma ^*_I \times 100$ as a performance measure, where $\varphi_I$ is the value of the upper bound  and $\varsigma ^*_I$ is the value of the best lower bound  for  $I$.

Similarly, to evaluate the quality of a lower bound for an instance $I$, we use the percentage gap between the  best known upper bound for $I$ and the lower bound i.e. $(\varphi^*_I - \varsigma_I)/\varsigma_I \times 100$ as a performance measure, where $\varphi^*_I$ is the value of the best known upper bound  for $I$ and $\varsigma_I$ is the value of the lower bound for $I$.

In order to compare the quality of different upper (lower) bounds we use performance profiles, in which,  for each upper (lower) bound and a value $g$ on the horizontal axis, we plot the percentage of instances that have the percentage gap to the best lower (upper) bound less than or equal to $g$\%.

\subsection{Experimental Framework}\label{sec:exp_framework}
The MILP formulations and heuristics are implemented in C++ and run on a Dell PowerEdge R710 with dual hex core 3.06GHz Intel Xeon X5675 processors and 96GB RAM running Red Hat Enterprise Linux 6.  IBM ILOG CPLEX v12.5 is used in deterministic mode with a single thread.  For each  formulation we investigated the performance of different root algorithms,  i.e. algorithms to solve the linear programming relaxation of the problem, provided in CPLEX. The primal simplex method performed the best overall and hence was  used as the root algorithm for each formulation. For all formulations a time limit of 120 minutes (7200 seconds) to solve each randomly generated instance was imposed. For the much larger HVCC instances, we allowed twice as much time: the time limit was set to 240 minutes (14400 seconds). Each formulation was given an initial feasible solution in which a job $j$ starts at time given by $\lfloor (r_j +d_j -p_j)/2 \rfloor$. All other CPLEX parameters were set to their default values.

In our computational study, we compare the quality of upper bounds given by CTIP, TDIP(RD) and TDIP(TI) formulations. We extract two upper bounds from each: (i) the LP relaxation value, denoted by  LP-CTIP, LP-TDIP(RD) and LP-TDIP(TI) respectively, and (ii) the best bound given by the model at the end of the time limit, UB-CTIP, UB-TDIP(RD) and UB-TDIP(TI).

We also compare the quality of the lower bounds provided by CTIP and TDIP(TI)$^B$  at the end of the time limit, denoted by LB-CTIP and LB1 respectively. In addition to these, we also consider the following lower bounds for the computational analysis.
\begin{itemize}
\item CoM (Proj): the value of the feasible solution obtained by applying the CoM (Proj) heuristic to the LP-relaxation solution of TDIP(RD)
\item CoM-LP$\tau$ (Proj-LP$\tau$): the value of the feasible solution obtained by applying the CoM (Proj) heuristic to the LP relaxation solution at the active node with best upper bound found after solving TDIP(RD) for $\tau$ seconds.
\item CoM-FS$\tau$ (Proj-FS$\tau$): the value of the feasible solution obtained by applying the CoM (Proj) heuristic to the best feasible solution found by TDIP(RD) within $\tau$ seconds.
\item Max of All: the best feasible solution generated by any of the above heuristics.
\end{itemize}
For the randomly generated instances we use $\tau = 300$, but for the much larger HVCC instances, we use $\tau=1800$.

We note here that throughout the next section all averages are taken to be geometric means (unless otherwise stated). Since percentage gaps in bounds can be zero, we use the {\em shifted} geometric mean with a shift of 1: the shifted geometric mean  of values $x_1,x_2,\dots,x_n$ with shift $s$ is defined as $\left(\prod_{i=1}^n(x_i+s)\right)^{1/n}-s$.

\subsection{Results}

We begin by comparing the quality of bounds given by  CTIP and TDIP formulation on a modified smaller subset of the test data set, in which, for each instance in the original data set with storage capacity 5 (ten instances for each network giving a total of 80 instances), we discard all jobs with deadline greater than 60 and take the time horizon to be $T=60$. The average number of jobs over the resulting ten instances for each network are shown in the second column of Table~\ref{tab:comp_CTIP}. The quality of bounds produced on these modified instances can be compared by observing Figures~\ref{fig:lb_compare_ctip} and~\ref{fig:ub_compare_ctip}, with summary statistics given in Table~\ref{tab:comp_CTIP}. Run time statistics are given in Table~\ref{tab:runtimes_modifieddata}. We first explain what is shown in each figure or table, and then summarize our findings from these results.

Figure~\ref{fig:lb_compare_ctip} gives a performance profile for the percentage gap of each of the CTIP and TDIP model lower bounds (LB-CTIP and LB1 respectively). It also gives a dot plot, with two columns of dots for each network, one column for each for the CTIP and TDIP lower bounds. There are ten dots in each column, each plotted with $y$-axis value given by the percentage gap of the lower bound produced by that column's model on the corresponding instance. Each column also includes a box, plotted at a $y$-axis value given by the (shifted) geometric mean of the percentage gap produced by the model for that column over the ten instances represented. Figure \ref{fig:ub_compare_ctip} gives two plots with three performance profiles in each. The first provides the percentage gaps of each of the LP relaxation values of the CTIP, TDIP(RD) and TDIP(TI) formulations; the second provides the percentage gaps that were found by solving each of the CTIP, TDIP(RD) and TDIP(TI) MIPs with the given run time limit. Figure~\ref{fig:ub_compare_ctip} includes a dot plot similar to that in Figure~\ref{fig:lb_compare_ctip}, but with three columns per network rather than two, one for each of the upper bounds UB-CTIP, UB-TDIP(RD) and UB-TDIP(TI).

Table~\ref{tab:comp_CTIP} provides summary statistics for the performance of the CTIP and TDIP model bounds. Its third and eighth columns report the number of instances, out of ten, for which best feasible solution found by the CTIP and TDIP(TI)$^B$ MIPs within the run time limit is known to be optimal (has value equal to the best upper bound found for the instance with any model, i.e. has zero percentage gap), respectively. The table's fourth and fifth column report information ``internal'' to the CTIP MIP solution process, showing the number of instances (out of ten) for which the CTIP model was able to prove optimality of its feasible solution within the time limit, and the average (across the ten instances) of the gap between its best upper and best lower bounds at the end of the time limit, reported as a percentage. The table's sixth column reports the average percentage gap of the CTIP lower bound (LB-CTIP, calculated using the best upper bound found by any model). Similarly the seventh column reports the average percentage gap of the CTIP upper bound (UB-CTIP). The ninth column shows the average percentage gap of the lower bound produced by the TDIP model (i.e., LB1, from TDIP(TI)$^B$). The tenth and eleventh columns respectively report the average percentage gap of the two TDIP upper bounds, UB-TDIP(RD) and UB-TDIP(TI). Recall that all averages are shifted geometric means and that percentage gaps are all calculated with respect to the best bound produced by any model, with the exception of those in the fifth column, which are as explained above.

Summary statistics for the run times of CTIP and TDIP models on the modified dataset are given in Table~\ref{tab:runtimes_modifieddata}. For each model, the minimum, maximum and average (geometic mean) run time, in seconds, across the ten instances for each network, are shown.

\begin{figure}[htbp]
 \begin{minipage}{.45\linewidth}
     \bigskip
      \bigskip
   \includegraphics[width=\textwidth]{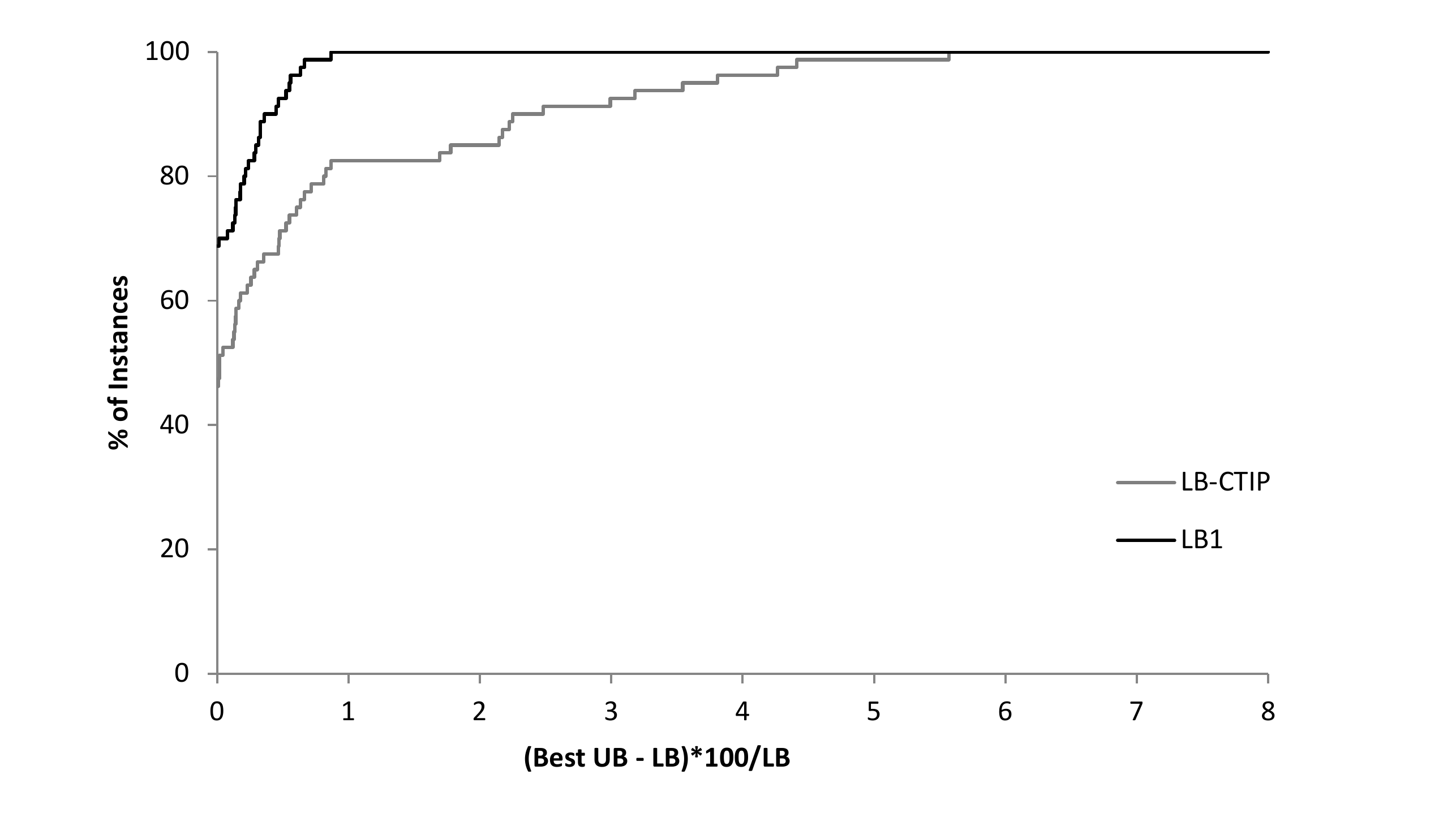}
     \end{minipage}\hfill
    \begin{minipage}{.45\linewidth}
   \bigskip
   \bigskip
  \includegraphics[width=\textwidth]{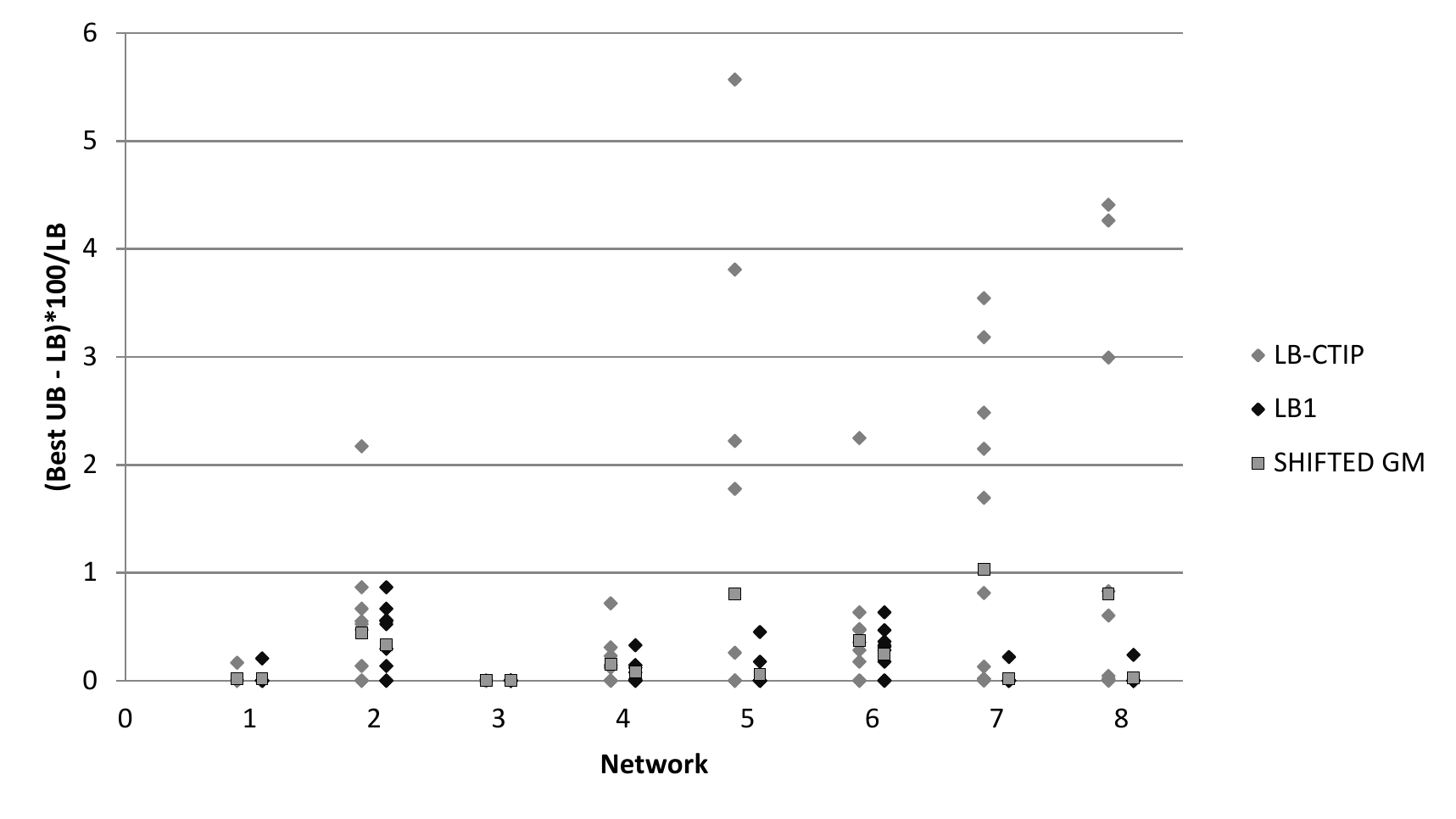}
    \end{minipage}
  \caption{Comparison of lower bounds on the modified data set.}\label{fig:lb_compare_ctip}
\end{figure}

\begin{figure}
   \centering
  \begin{minipage}{.45\linewidth}
      \bigskip
       \bigskip
    \includegraphics[width=\textwidth]{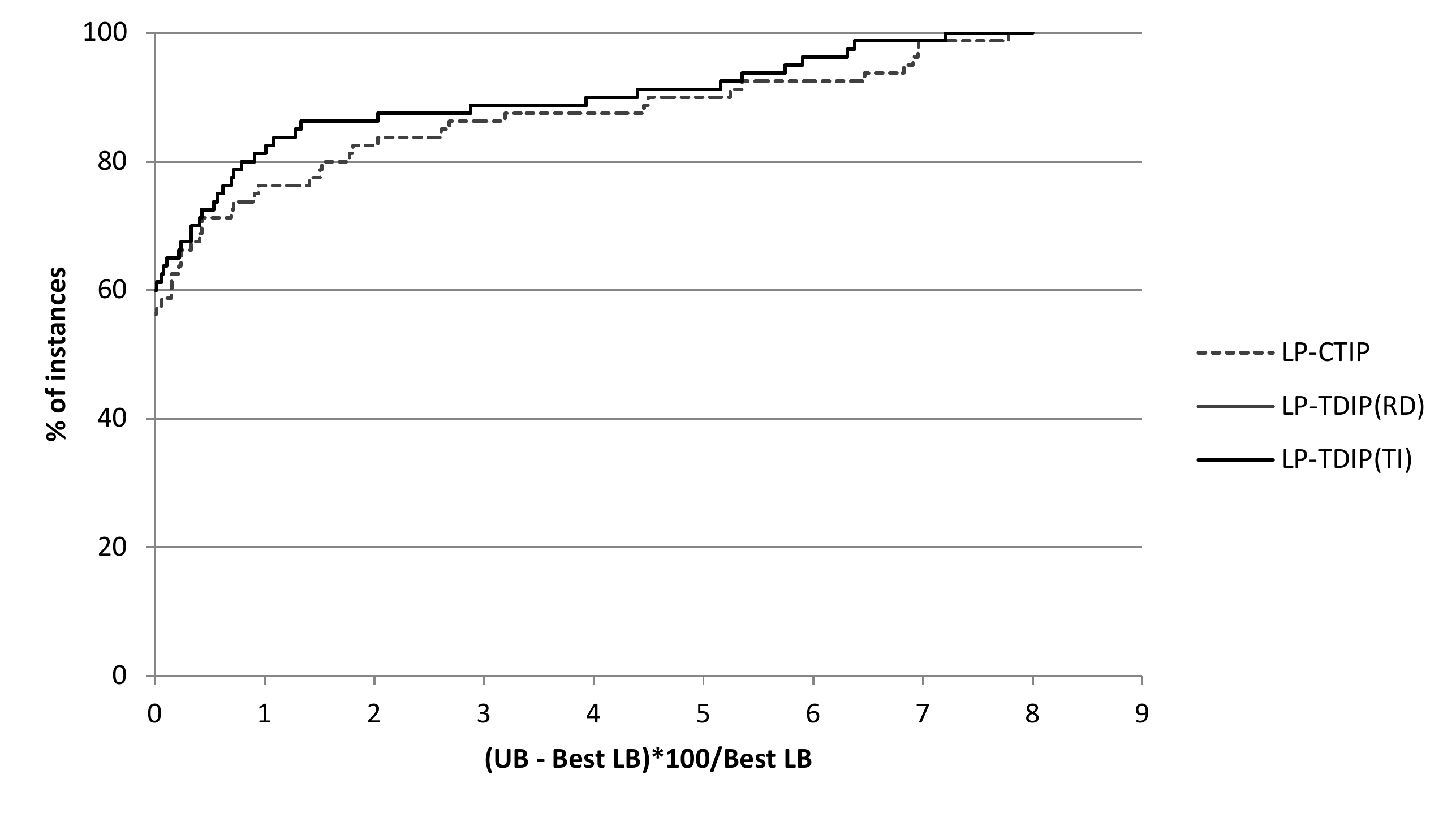}
      \end{minipage}\hfill
     \begin{minipage}{.45\linewidth}
    \bigskip
    \bigskip
   \includegraphics[width=\textwidth]{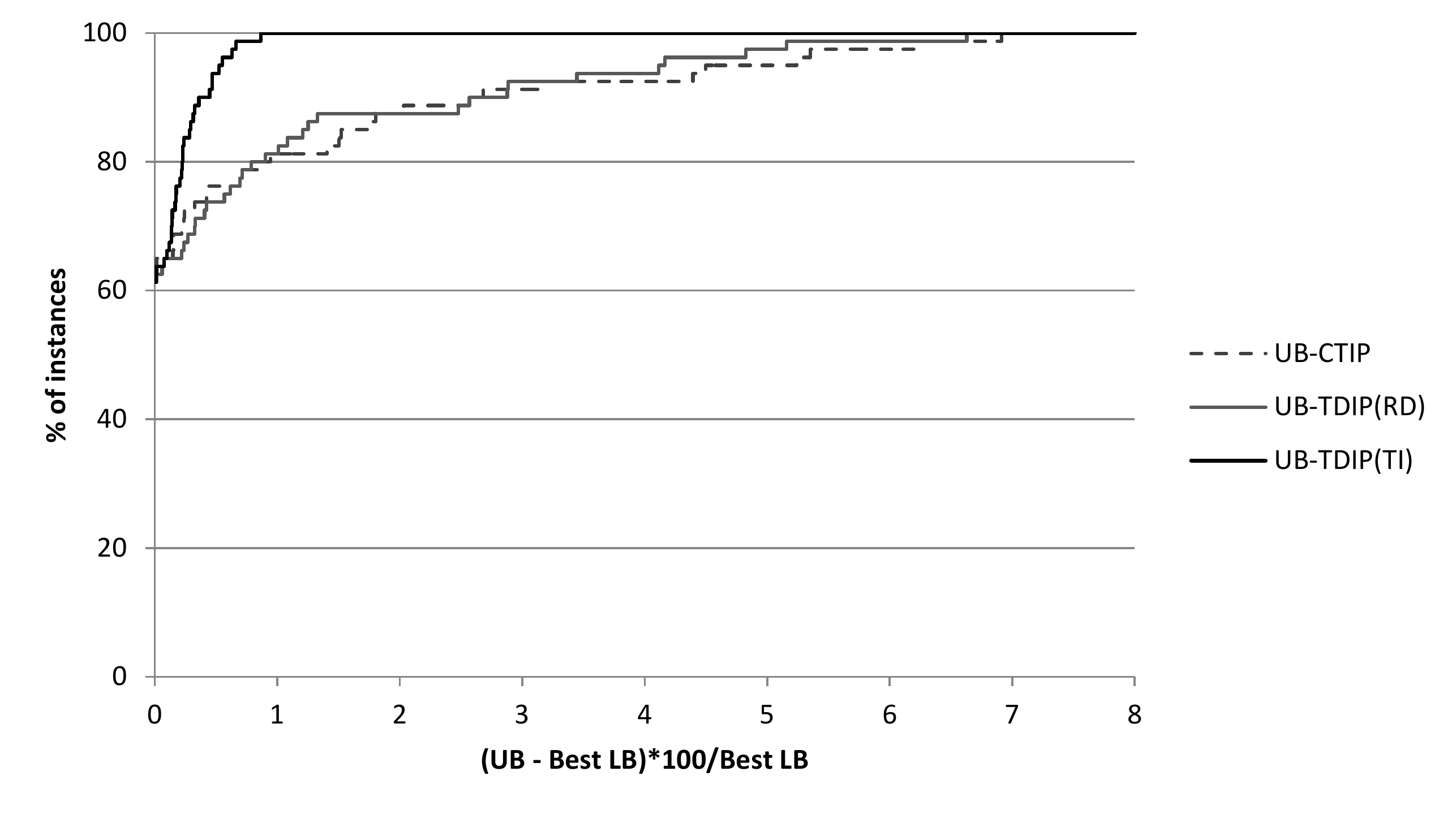}
     \end{minipage}
      \begin{minipage}{.45\linewidth}
         \bigskip
         \bigskip
        \includegraphics[width=\textwidth]{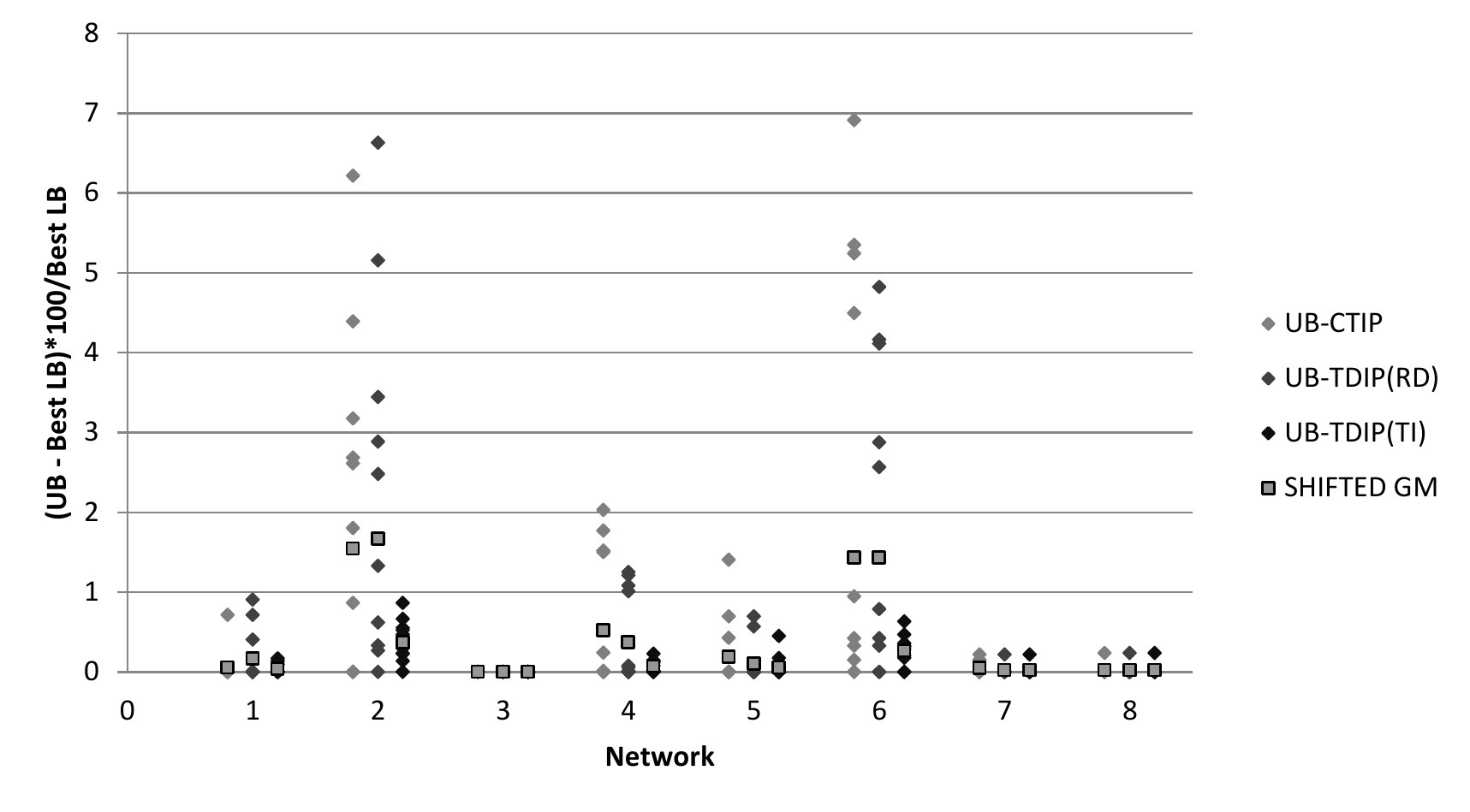}
          \end{minipage}
  \caption{Comparison of upper bounds on the modified data set.}\label{fig:ub_compare_ctip}
\end{figure}

\begin{sidewaystable}[htbp]
  \centering
\renewcommand{\arraystretch}{1.2}
{\small
     \begin{tabular}{ccccccccccc}
         \toprule
         \multicolumn{2}{c}{Network} & \multicolumn{5}{c}{CTIP} & \multicolumn{2}{c}{TDIP(TI)$^B$}   & \multicolumn{1}{c}{TDIP(RD)}  &  \multicolumn{1}{c}{TDIP(TI)}\\
\cmidrule( r){1-2}\cmidrule(lr){3-7}\cmidrule(lr){8-9}\cmidrule(lr){10-10}\cmidrule(l ){11-11}
        No.  			& Avg.  	 & \#   & \#  & \%  gap UB-CTIP & \%gap Best UB  &\%
        gap UB-CTIP & \#  & \% gap Best UB   	 & \% gap UB-TDIP  & \% gap UB-TDIP \\
	& \# Jobs & Opt	 & CTIP Opt & \& LB-CTIP & \& LB-CTIP & \& Best LB &  Opt & \& LB1 & (RD)
        \& Best LB & (TI)  \& Best LB  \\
\cmidrule( r){1-2}\cmidrule(lr){3-7}\cmidrule(lr){8-9}\cmidrule(lr){10-10}\cmidrule(l ){11-11}
     	1     & 7.93 & 9  & 9   & 0.06 &  0.02 & 0.06  & 9     & 0.02   & 0.16 & 0.04 \\
        2     & 11.89 & 3 & 3    & 1.68 &  0.44 & 1.54  & 3 & 0.33   & 1.66 & 0.37 \\
        3     & 12.08 & 10 & 10   & 0.00 &  0.00 & 0.00  & 10  & 0.00   & 0.00 & 0.00 \\
        4     & 20.93 & 4  & 4     & 0.63 &  0.15 &0.52  & 4     & 0.08   & 0.37 & 0.07 \\
        5     & 28.79 & 5  &5     &0.90 &  0.80 &0.19  & 8     & 0.05   & 0.10 & 0.05 \\
        6     & 20.71 & 3  &2     &1.62 &  0.37 &1.43  & 3     & 0.24   & 1.43 & 0.26 \\
        7     & 38.48 & 1  &1     & 1.05 &  1.03 & 0.05  & 9     & 0.02   & 0.02 & 0.02 \\
        8     & 56.14 & 2  &2     &0.80 &  0.80 &0.02  & 9     & 0.02  & 0.02 & 0.02 \\     \bottomrule
 \end{tabular}}
\caption{Comparison of CTIP with TDIP on the modified data set.}\label{tab:comp_CTIP}
\bigskip\bigskip

{\small
\begin{tabular}{@{}cS[table-format=1.2]S[table-format=1.2]S[table-format=1.2]S[table-format=3.2]S[table-format=4.2]S[table-format=4.2]S[table-format=2.2]S[table-format=1.2]S[table-format=1.2]S[table-format=1.2]S[table-format=1.2]S[table-format=1.2]S[table-format=1.2]S[table-format=2.2]S[table-format=4.2]S[table-format=1.2]S[table-format=1.2]S[table-format=4.2]@{}}
    \toprule
    Network & \multicolumn{6}{c}{CTIP} & \multicolumn{12}{c}{TDIP}\\
    \cmidrule(lr){2-7}\cmidrule(l ){8-19}
    No. & \multicolumn{3}{c}{LP-CTIP} & \multicolumn{3}{c}{CTIP} & \multicolumn{3}{c}{TDIP(RD)} & \multicolumn{3}{c}{LP-TDIP(TI)}       & \multicolumn{3}{c}{TDIP(TI)}        & \multicolumn{3}{c}{TDIP(TI)$^B$} \\
    \cmidrule( r){1-1}\cmidrule(lr){2-4}\cmidrule(lr){5-7}\cmidrule(lr){8-10}\cmidrule(lr){11-13}\cmidrule(lr){14-16}\cmidrule(lr){17-19}
          & {min}  & {avg}   & {max}   & {min}   & {avg}   & {max}   & {min}   & {avg}   & {max}   & {min}   & {avg}   &
          {max}   & {min}   & {avg}   & {max}   & {min}   & {avg}   & {max} \\
\cmidrule(lr){2-4}\cmidrule(lr){5-7}\cmidrule(lr){8-10}\cmidrule(lr){11-13}\cmidrule(lr){14-16}\cmidrule(lr){17-19}
    1     & 0.01  & 0.01  & 0.02  & 0.03  & 3.45  & 7200  & 0.01  & 0.01  & 0.02  & 0.01  & 0.02  & 0.03  & 0.06  & 0.35  & 33.51 & 0.06  & 0.20  & 3.68 \\
    2     & 0.03  & 0.21  & 0.73  & 0.11  & 1061.01 & 7200  & 0.01  & 0.03  & 0.05  & 0.02  & 0.05  & 0.10  & 0.07  & 3.46  & 88.44 & 0.14  & 1.00  & 11.78 \\
    3     & 0.03  & 0.06  & 0.20  & 0.07  & 3.40  & 112.92 & 0.01  & 0.02  & 0.03  & 0.02  & 0.03  & 0.06  & 0.09  & 0.13  & 0.28  & 0.09  & 0.14  & 0.23 \\
    4     & 0.11  & 0.19  & 0.83  & 51.69 & 2078.91 & 7200  & 0.03  & 0.06  & 0.23  & 0.05  & 0.08  & 0.28  & 0.17  & 4.26  & 1979.58 & 0.19  & 3.01  & 1136.56 \\
    5     & 0.33  & 0.63  & 3.46  & 0.59  & 1937.75 & 7200  & 0.04  & 0.08  & 0.24  & 0.08  & 0.13  & 0.23  & 0.36  & 4.92  & 688.24 & 0.37  & 3.22  & 132.61 \\
    6     & 0.18  & 0.41  & 1.90  & 104.74 & 4319.72 & 7200  & 0.03  & 0.10  & 0.41  & 0.03  & 0.11  & 0.24  & 0.15  & 17.12 & 1798.90 & 0.19  & 5.57  & 70.86 \\
    7     & 0.29  & 0.94  & 4.09  & 1.97  & 3171.03 & 7200  & 0.09  & 0.13  & 0.31  & 0.10  & 0.15  & 0.45  & 0.34  & 2.79  & 7200  & 0.38  & 2.07  & 448.22 \\
    8     & 0.92  & 2.25  & 5.98  & 2.56  & 1472.38 & 7200  & 0.15  & 0.28  & 1.29  & 0.14  & 0.25  & 0.59  & 0.45  & 2.64  & 7200  & 0.55  & 2.54  & 759.88 \\
    \bottomrule
    \end{tabular}}
\caption{Run time statistics of MIPs on the modified data set, reported in
  seconds.}\label{tab:runtimes_modifieddata}  
\end{sidewaystable}%

Our first observation from these results is that the CTIP MIP model struggles to solve to optimality. Even on these very small modified instances, it times out on 44 of the 80 instances. By comparison, on these instances, almost none of the TDIP formulations reached the run time limit: the only exception was the TDIP(TI) formulation, which timed out on 2 of the 20 instances for Networks 7 and 8.

In terms of the quality of lower bounds produced, the CTIP model was clearly outperformed by the TDIP(TI)$^B$ model. This can be immediately observed by comparing their performance profiles in Figure \ref{fig:lb_compare_ctip}, in which the profile for LB1 remains well above and to the left of that for LB-CTIP, until they converge at a percentage gap of around 6\%. For all instances LB1 is within 2\% of the best upper bound whereas LB-CTIP is within 2\% for only 80\% of instances. In the dot plot, we see the second column (for LB1) is typically much shorter, with the first column (for LB-CTIP) showing a spread with several instances' dots plotted higher than the height of the second column. From Table~\ref{tab:comp_CTIP}, we see that only 46.25\% instances are solved to optimality by CTIP, i.e LB-CTIP is optimal, versus 68.75\% for LB1. For every network the average percentage gap of LB1 is smaller than that of LB-CTIP, with the relative difference between the gaps increasing with increasing size of the instances. From Table~\ref{tab:runtimes_modifieddata} we see that run times for TDIP(TI)$^B$ are also usually substantially shorter (in some cases by several orders of magnitude) than those for CTIP.

In regard to upper bounds, the TDIP models also offer a much better trade-off for quality versus run time when compared to the CTIP model. From the first performance profile plot in Figure \ref{fig:ub_compare_ctip}, we see that the quality of upper bound from the CTIP model's LP relaxation is similar to that of TDIP(TI), but as can be seen from Table~\ref{tab:runtimes_modifieddata}, the run times for the latter are generally lower, and more so for the larger instances. From the second performance profile plot in Figure \ref{fig:ub_compare_ctip}, we see that the quality of upper bound from the CTIP MIP is similar to that of TDIP(RD) (and noticeably worse than that of TDIP(TI)), but as can be seen from Table~\ref{tab:runtimes_modifieddata}, the run times for the TDIP(RD) (and TDIP(TI)) MIPs are generally lower, by several orders of magnitude in the case of TDIP(RD).

We conclude that the CTIP model is primarily of theoretical value, allowing us to provide insights about the nature of optimal solutions but struggling to solve problems in practice. From the experiments on instances that are much smaller than those in the original data set, it seems very unlikely that the CTIP model could perform well on realistic sized instances, and appears to be outperformed in all respects by the TDIP models. Thus in what follows do not report results for the CTIP model on the full sized and real world instances; we now focus on the performance of the TDIP models and the heuristics on these instances.

Figure \ref{fig:gap_from_BUB_heuristics} provides performance profiles for percentage gaps of the lower bounds, with two plots: one for the instances with Networks 1-4, which are markedly easier, and the other for the instances with Networks 5-8. Summary statistics for the performance of the lower bounds on the randomly generated instances are given in Table~\ref{tab:gaps_LB}, which shows the minimum, maximum and average (shifted geometric mean) percentage gaps of each heuristic, over all ten instances with each network and each storage capacity level. The same statistics taken over all forty instances with each network are reported in bold font. Table~\ref{tab:gaps_UB} is similar, but shows statistics for the performance of the upper bounds. Percentage gaps of lower and upper bounds on the HVCC instances are shown Tables~\ref{tab:gaps_LB_HVCC} and~\ref{tab:gaps_UB_HVCC} respectively. Run time statistics are reported in Table~\ref{tab:run_times} for the randomly generated instances and in Table~\ref{tab:run_times_HVCC} for the HVCC instances. In the former case, the columns labelled ``\# time limit'' record the number of instances on which the corresponding model timed out, out of the number represented by the row statistics (ten or forty).

First, we observe that all TDIP formulations struggled to solve these instances to optimality. None of the TDIP(RD), TDIP(TI) or TDIP(TI)$^B$ formulations solved the HVCC instances to optimality within the time limit. As can be seen from Table~\ref{tab:run_times}, the formulations all reached the time limit for almost all randomly generated instances with networks other than Networks 1 and 3. For Network 3 all instances solved to optimality with all formulations, while for Network 1, TDIP(RD) solved to optimality for all instances, but TDIP(TI) only solved 17, while TDIP(TI)$^B$ could only solve 14. The TDIP(RD) formulation solved more instances to optimality than any other, but this was still only 91 instances out of 320. We note that instances seem to get more difficult as the storage capacity decreases, as well as with increasing network size and number of jobs.

Fortunately, even without solving to optimality, the TDIP(RD) and TDIP(TI) formulations provide quite good upper bounds. As can be seen from Table~\ref{tab:gaps_UB}, UB-TDIP(TI) appears to provide the best upper bounds overall on randomly generated instances, with average percentage gaps less than 2\% for all networks, and less than 2.3\% for all networks and storage capacity levels, except for those with Network 6, which has quite large gaps, averaging 6.87\%. For Network 6, UB-TDIP(RD) appears noticeably better, and is very close to UB-TDIP(TI) on the other networks. For the HVCC instances reported in Table~\ref{tab:run_times_HVCC}, UB-TDIP(TI) is best for the 2010 instance, but UB-TDIP(RD) is best for 2011. For most randomly generated instances, the LP relaxations actually provide upper bounds that are nearly as good as those given by the MIPs, and of course take much less time (as per Table~\ref{tab:run_times}). The exception is Network 2, for which there is a noticeable improvement in the quality of the upper bound resulting from solving the MIP rather than just its LP relaxation. Noticeable improvements are also observed for the HVCC instances, particularly the 2010 instance. From these results it is difficult to ``pick a winner'' between UB-TDIP(TI) and UB-TDIP(RD), but it does appear that UB-TDIP(RD) does relatively well in cases where $|\mathcal D|$ is smaller: the 2011 HVCC instance has $|\mathcal D|$ value quite a bit smaller than that for 2010, and Network 6 has the smallest $|\mathcal D|$ of all of the harder networks (5-8). One possibility is that when there are fewer unique job start and end times, the more compact TDIP(RD) formulation offers a better trade-off between MIP solvability and bound quality than does TDIP(TI).

The lower bounds show quite a bit more diversity than the upper bounds in terms of offering a trade-off of quality versus run time. It is clear that with long run times, LB1 (found by solving the MIP TDIP(TI)$^B$) is by a large margin the best lower bound on randomly generated instances: it gives feasible solutions that can be proved to be within 2.06\% of optimality for all instances with Networks 1-4 and gives feasible solutions within 5\% of optimality on more than 80\% of instances with Networks 5-8 (see Table~\ref{tab:gaps_LB} and Figure~\ref{fig:gap_from_BUB_heuristics}). No other lower bound is better than LB1 on any randomly generated instance. The situation on HVCC instances (Table~\ref{tab:gaps_LB_HVCC}) is somewhat different: although LB1 is very close to the best on the 2010 instance, it is very far from best on the 2011 instance. Here the CoM and Proj heuristic methods come into their own.

The performance of all variants of the CoM and Proj heuristics, as described in Section \ref{sec:exp_framework}, on the smaller randomly generated instances (those with Networks 1-4) are shown in Figure~\ref{fig:gap_from_BUB_heuristics}. From this plot we observe that the methods based on the initial LP relaxation solution of TDIP(RD) perform very poorly: although they are very fast to run, with only a small amount of extra computing time (5 minutes), very substantial improvements in the bound quality can be made. The same conclusion can be reached by comparing the columns for ``Max(CoM,Proj)'' with those for ``Max(Proj-LP300,Proj-FS300)'' and ``Max(CoM-LP,CoM-FS300)'' in Table~\ref{tab:gaps_LB}; the latter show far smaller percentage gaps. Thus in the plot for Networks 5-8 in Figure~\ref{fig:gap_from_BUB_heuristics}) we focus on the heuristics applied to either best upper bound LP solutions or best lower bound integer feasible solutions found after 5 minutes of computing time. Here it can be seen that the LP solutions seem to be a better basis for the heuristics than the integer feasible solutions, with both Proj-LP300 and CoM-LP300 profiles lying above the Proj-FS300 and CoM-FS300 profiles, with some margin. This can be explained by the fact that for these instances the first LP relaxation takes quite a long time to solve, (around 200 seconds on average), hence for most instances the MIP solver could not improve on the initial feasible solution provided for TDIP(RD) within the 300 second limit. We also note that the Proj heuristic seems slightly better than CoM: the difference is small when based on the integer feasible solution, but is noticeable when the heuristics are based on the LP solution. This can also be observed by comparing the columns for ``Max(Proj-LP300,Proj-FS300)'' with those of ``Max(CoM-LP,CoM-FS300)'' in Table~\ref{tab:gaps_LB}. The situation on HVCC instances is quite different. Here the heuristics based on the best integer feasible solutions at the (longer) 1800 second mark are actually better than those based on the LP solution giving the best upper bound at that time, by quite a large margin. Although the time limit at which to extract the lower bound relative to the size or difficulty of the instance is difficult to calibrate, this would suggest that the integer feasible solution offers an increasingly better basis for the heuristics as longer TDIP(RD) MIP run times are allowed.

Since the CoM and Proj heuristics are computationally very cheap to run, and one variant is not consistently better than another on all instances, we also consider running all variants after the TDIP(RD) MIP lower bound time limit is reached (``Max of All''). As can be seen from the performance profiles on the harder randomly generated instances (Networks 5-8) in Figure~\ref{fig:gap_from_BUB_heuristics}, and by comparing the columns for ``Max of All'' in Table~\ref{tab:gaps_LB} with those for the other heuristic variants, this gives noticeable improvements over the results of any one (or pair of) heuristics alone, at virtually no extra computational cost. Whilst the quality of the lower bound produced by ``Max of All'' is still nowhere near as good as LB1 on randomly generated instances, it comes at far less computational cost, and has the added advantage of ``controllability'' of the computational cost via the time at which the solutions on which to base the heuristics are extracted. Indeed, it gives the best quality lower bounds on the HVCC instances when these solutions are extracted at 1800 seconds.


\begin{figure}[htb]
\centering
  \begin{minipage}{.48\linewidth}
\includegraphics[width=\textwidth]{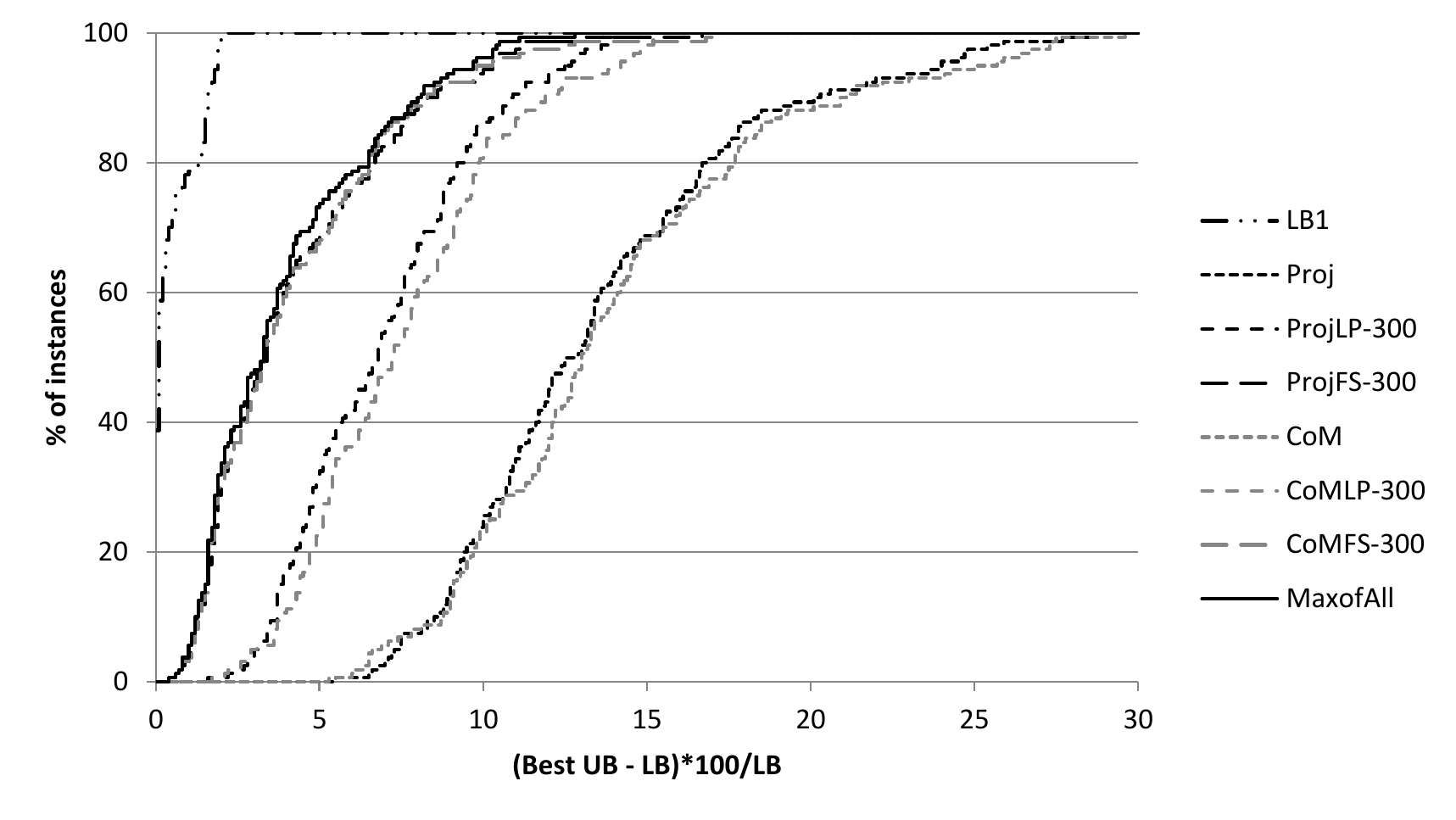}
\caption*{Networks 1-4}
  \end{minipage}\hfill
  \begin{minipage}{.48\linewidth}
\includegraphics[width=\textwidth]{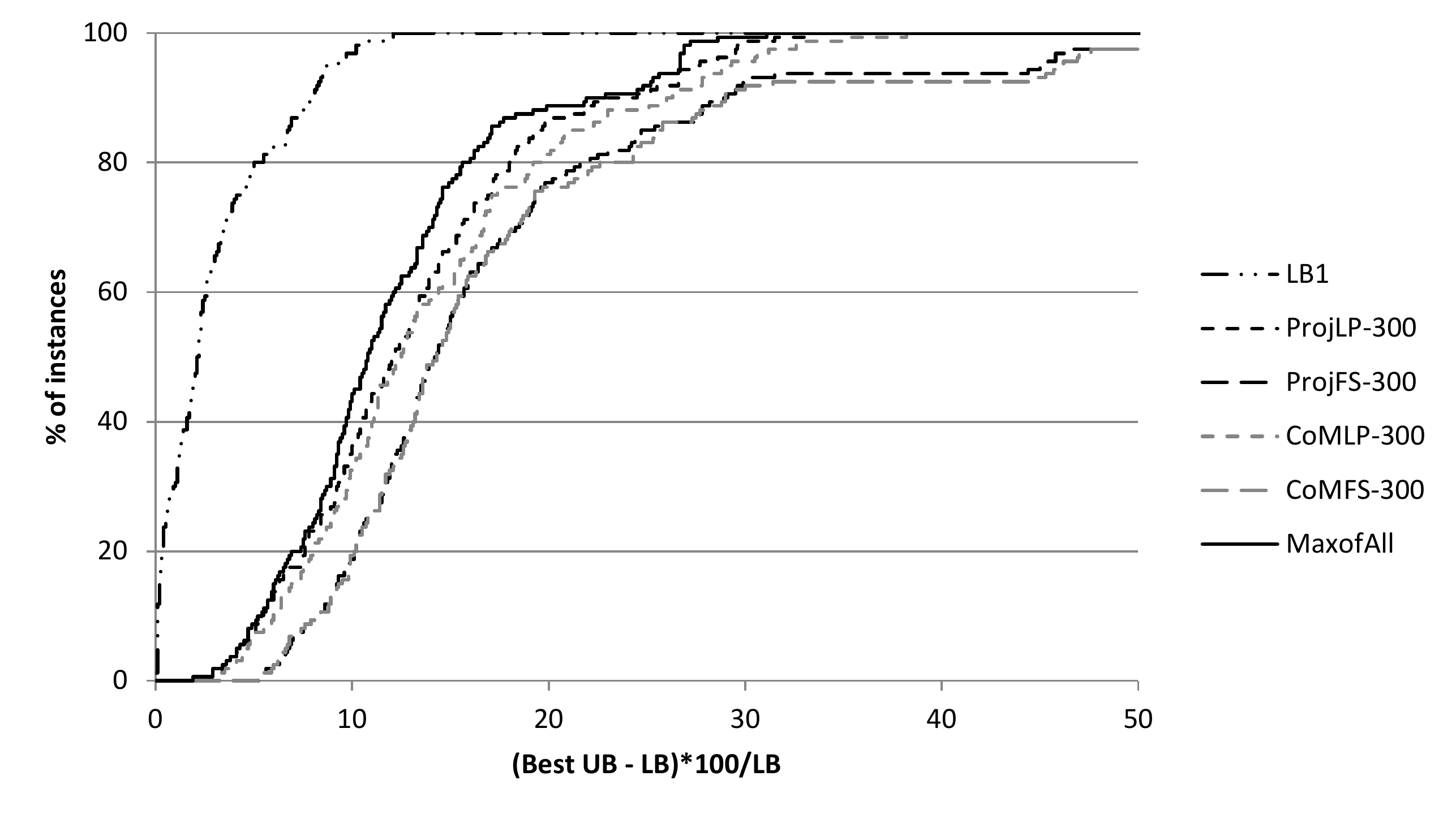}
\caption*{Networks 5-8}
  \end{minipage}
  \caption{Performance profiles for lower bounds on randomly generated instances.}
  \label{fig:gap_from_BUB_heuristics}
\end{figure}


\begin{table}
	\centering
    \centering
        \scalebox{0.72}{       
   \begin{tabular}{@{}ccccccccccccccccc@{}}
        \toprule
        Network & Storage & \multicolumn{3}{c}{\qquad \%gap Best UB  \& \qquad} & \multicolumn{3}{c}{\qquad \%gap Best UB  \& \qquad} & \multicolumn{3}{c}{\qquad \%gap Best UB \& \qquad } & \multicolumn{3}{c}{\qquad\%gap Best UB \& \qquad} & \multicolumn{3}{c}{\qquad \%gap Best UB\ \ \& } \\
       	& Capacity  & \multicolumn{3}{c}{ \quad LB1   } & \multicolumn{3}{c}{ \quad Max(CoM,Proj)   } & \multicolumn{3}{c}{ Max(CoM-LP300,   } &\multicolumn{3}{c}{ Max(Proj-LP300,   } &\multicolumn{3}{c}{ Max of All  }\\
       	& & \multicolumn{3}{c}{  } &\multicolumn{3}{c}{  } &\multicolumn{3}{c}{CoM-FS300)}
        &\multicolumn{3}{c}{Proj-FS300)}&\multicolumn{3}{c}{  } \\ 
\cmidrule( r){1-2}\cmidrule(lr){3-5}\cmidrule(lr){6-8}\cmidrule(lr){9-11}\cmidrule(lr){12-14}\cmidrule(lr){15-17}
          &       & min   & avg   & max   & min   & avg   & max   & min   & avg   & max   & min   &
          avg   & max   & min   & avg   & max \\
\cmidrule(lr){3-5}\cmidrule(lr){6-8}\cmidrule(lr){9-11}\cmidrule(lr){12-14}\cmidrule(lr){15-17}
    1     & 5     & 0.00  & 0.08  & 0.27  & 5.93  & 11.44 & 16.28 & 1.64  & 3.15  & 5.71  & 1.61  & 3.14  & 5.90  & 1.61  & 3.06  & 5.71 \\
          & 10    & 0.00  & 0.04  & 0.23  & 6.48  & 12.14 & 17.41 & 1.56  & 2.97  & 5.25  & 1.62  & 3.13  & 5.33  & 1.56  & 2.96  & 5.25 \\
          & 15    & 0.00  & 0.02  & 0.11  & 5.24  & 11.63 & 15.47 & 1.68  & 3.16  & 5.07  & 1.54  & 3.22  & 5.39  & 1.54  & 3.12  & 5.07 \\
          & 20    & 0.00  & 0.01  & 0.03  & 7.00  & 10.76 & 14.55 & 1.22  & 2.90  & 5.55  & 1.33  & 2.96  & 5.71  & 1.22  & 2.89  & 5.55 \\
 \cmidrule( r){2-2}\cmidrule(lr){3-5}\cmidrule(lr){6-8}\cmidrule(lr){9-11}\cmidrule(lr){12-14}\cmidrule(lr){15-17}
    \multicolumn{2}{r}{\textbf{Average\ \ }} & \textbf{0.00} & \textbf{0.04} & \textbf{0.27} & \textbf{5.24} & \textbf{11.48} & \textbf{17.41} & \textbf{1.22} & \textbf{3.04} & \textbf{5.71} & \textbf{1.33} & \textbf{3.11} & \textbf{5.90} & \textbf{1.22} & \textbf{3.00} & \textbf{5.71} \\
    \midrule
          &       & {}   &  &  & &  &  &  &   &   &   &  & & & & \\[-1.5ex]
    2     & 5     & 0.22  & 1.37  & 2.06  & 14.72 & 18.83 & 25.34 & 5.53  & 7.84  & 12.73 & 4.08  & 7.52  & 12.82 & 4.08  & 7.37  & 12.73 \\
          & 10    & 0.22  & 1.31  & 1.93  & 11.92 & 18.92 & 28.53 & 4.36  & 7.10  & 11.31 & 5.29  & 7.47  & 10.30 & 4.36  & 7.02  & 10.30 \\
          & 15    & 0.16  & 1.22  & 1.87  & 13.34 & 18.45 & 25.89 & 4.01  & 6.90  & 10.34 & 3.92  & 7.23  & 10.96 & 3.92  & 6.86  & 10.29 \\
          & 20    & 0.12  & 1.08  & 1.80  & 12.39 & 18.30 & 27.49 & 3.80  & 6.57  & 10.24 & 2.89  & 6.69  & 10.21 & 2.89  & 6.37  & 10.21 \\
  \cmidrule( r){2-2}\cmidrule(lr){3-5}\cmidrule(lr){6-8}\cmidrule(lr){9-11}\cmidrule(lr){12-14}\cmidrule(lr){15-17}
    \multicolumn{2}{r}{\textbf{Average\ \ }} & \textbf{0.12} & \textbf{1.24} & \textbf{2.06} & \textbf{11.92} & \textbf{18.62} & \textbf{28.53} & \textbf{3.80} & \textbf{7.09} & \textbf{12.73} & \textbf{2.89} & \textbf{7.22} & \textbf{12.82} & \textbf{2.89} & \textbf{6.89} & \textbf{12.73} \\
    \midrule
          &       & {}   &  &  & &  &  &  &   &   &   &  & & & & \\[-1.5ex]
    3     & 5     & 0.00  & 0.00  & 0.00  & 8.01  & 11.14 & 15.87 & 0.57  & 1.58  & 2.82  & 0.58  & 1.61  & 2.79  & 0.57  & 1.57  & 2.79 \\
          & 10    & 0.00  & 0.00  & 0.00  & 6.88  & 10.56 & 15.89 & 1.01  & 1.63  & 2.52  & 1.08  & 1.68  & 2.34  & 1.01  & 1.60  & 2.28 \\
          & 15    & 0.00  & 0.00  & 0.00  & 6.58  & 10.34 & 14.28 & 0.78  & 1.56  & 2.28  & 0.83  & 1.60  & 2.04  & 0.78  & 1.53  & 2.04 \\
          & 20    & 0.00  & 0.00  & 0.00  & 6.01  & 10.69 & 15.56 & 1.17  & 1.59  & 2.28  & 1.15  & 1.69  & 2.31  & 1.15  & 1.59  & 2.28 \\
    \cmidrule( r){2-2}\cmidrule(lr){3-5}\cmidrule(lr){6-8}\cmidrule(lr){9-11}\cmidrule(lr){12-14}\cmidrule(lr){15-17}
    \multicolumn{2}{r}{\textbf{Average\ \ }} & \textbf{0.00} & \textbf{0.00} & \textbf{0.00} & \textbf{6.01} & \textbf{10.68} & \textbf{15.89} & \textbf{0.57} & \textbf{1.59} & \textbf{2.82} & \textbf{0.58} & \textbf{1.64} & \textbf{2.79} & \textbf{0.57} & \textbf{1.57} & \textbf{2.79} \\
     \midrule
          &       & {}   &  &  & &  &  &  &   &   &   &  & & & & \\[-1.5ex]
    4     & 5     & 0.00  & 0.40  & 1.42  & 7.35  & 11.20 & 19.36 & 1.20  & 3.88  & 12.37 & 0.91  & 2.70  & 10.49 & 0.91  & 2.69  & 10.49 \\
          & 10    & 0.00  & 0.32  & 1.26  & 6.36  & 10.94 & 20.05 & 1.01  & 3.73  & 14.51 & 1.10  & 3.71  & 8.63  & 1.01  & 3.35  & 8.63 \\
          & 15    & 0.00  & 0.21  & 0.74  & 6.41  & 10.48 & 17.73 & 0.36  & 3.00  & 9.76  & 0.36  & 2.80  & 9.93  & 0.36  & 2.77  & 9.76 \\
          & 20    & 0.00  & 0.18  & 0.54  & 6.42  & 9.93  & 18.43 & 0.76  & 2.81  & 9.02  & 1.02  & 3.22  & 11.99 & 0.76  & 2.62  & 9.02 \\
   \cmidrule( r){2-2}\cmidrule(lr){3-5}\cmidrule(lr){6-8}\cmidrule(lr){9-11}\cmidrule(lr){12-14}\cmidrule(lr){15-17}
    \multicolumn{2}{r}{\textbf{Average\ \ }} & \textbf{0.00} & \textbf{0.27} & \textbf{1.42} & \textbf{6.36} & \textbf{10.63} & \textbf{20.05} & \textbf{0.36} & \textbf{3.33} & \textbf{14.51} & \textbf{0.36} & \textbf{3.09} & \textbf{11.99} & \textbf{0.36} & \textbf{2.85} & \textbf{10.49} \\
    \midrule
          &       & {}   &  &  & &  &  &  &   &   &   &  & & & & \\[-1.5ex]
    5     & 5     & 0.18  & 2.24  & 6.84  & 9.46  & 13.65 & 18.22 & 5.98  & 12.47 & 21.11 & 5.39  & 11.75 & 19.73 & 5.39  & 10.94 & 18.22 \\
          & 10    & 0.26  & 2.06  & 3.76  & 9.91  & 13.52 & 18.02 & 8.56  & 12.34 & 16.71 & 8.31  & 12.90 & 19.31 & 8.31  & 11.74 & 16.62 \\
          & 15    & 0.10  & 1.77  & 3.83  & 10.16 & 13.76 & 17.96 & 6.33  & 12.39 & 17.33 & 5.93  & 12.23 & 17.99 & 5.93  & 11.72 & 17.02 \\
          & 20    & 0.23  & 1.74  & 3.11  & 9.67  & 13.46 & 17.05 & 9.69  & 12.92 & 18.94 & 8.39  & 12.86 & 19.62 & 8.39  & 11.96 & 17.05 \\
   \cmidrule( r){2-2}\cmidrule(lr){3-5}\cmidrule(lr){6-8}\cmidrule(lr){9-11}\cmidrule(lr){12-14}\cmidrule(lr){15-17}
    \multicolumn{2}{r}{\textbf{Average\ \ }} & \textbf{0.10} & \textbf{1.94} & \textbf{6.84} & \textbf{9.46} & \textbf{13.60} & \textbf{18.22} & \textbf{5.98} & \textbf{12.53} & \textbf{21.11} & \textbf{5.39} & \textbf{12.42} & \textbf{19.73} & \textbf{5.39} & \textbf{11.58} & \textbf{18.22} \\
    \midrule
          &       & {}   &  &  & &  &  &  &   &   &   &  & & & & \\[-1.5ex]
    6     & 5     & 3.67  & 6.45  & 9.27  & 16.70 & 21.82 & 26.81 & 12.74 & 21.58 & 38.11 & 11.69 & 19.27 & 33.46 & 11.69 & 18.43 & 26.81 \\
          & 10    & 2.91  & 6.62  & 10.76 & 15.75 & 21.97 & 27.13 & 12.93 & 20.81 & 33.03 & 11.47 & 18.90 & 29.90 & 11.47 & 18.17 & 27.13 \\
          & 15    & 2.51  & 6.36  & 12.01 & 16.88 & 22.80 & 31.06 & 13.59 & 21.28 & 35.18 & 11.43 & 19.15 & 31.48 & 11.43 & 19.07 & 31.06 \\
          & 20    & 4.29  & 6.93  & 11.54 & 18.52 & 23.66 & 30.10 & 13.26 & 19.88 & 30.44 & 11.92 & 19.65 & 28.56 & 11.92 & 18.81 & 28.56 \\
   \cmidrule( r){2-2}\cmidrule(lr){3-5}\cmidrule(lr){6-8}\cmidrule(lr){9-11}\cmidrule(lr){12-14}\cmidrule(lr){15-17}
    \multicolumn{2}{r}{\textbf{Average\ \ }} & \textbf{2.51} & \textbf{6.59} & \textbf{12.01} & \textbf{15.75} & \textbf{22.55} & \textbf{31.06} & \textbf{12.74} & \textbf{20.88} & \textbf{38.11} & \textbf{11.43} & \textbf{19.24} & \textbf{33.46} & \textbf{11.43} & \textbf{18.62} & \textbf{31.06} \\
    \midrule
          &       & {}   &  &  & &  &  &  &   &   &   &  & & & & \\[-1.5ex]
    7     & 5     & 0.06  & 0.70  & 2.23  & 4.84  & 9.53  & 12.27 & 4.30  & 8.33  & 12.64 & 3.54  & 7.38  & 12.35 & 3.54  & 7.09  & 11.45 \\
          & 10    & 0.02  & 0.64  & 2.88  & 5.41  & 9.73  & 12.58 & 3.48  & 7.79  & 11.35 & 3.35  & 8.07  & 13.69 & 3.35  & 7.51  & 10.92 \\
          & 15    & 0.00  & 0.65  & 3.11  & 5.20  & 9.71  & 14.02 & 1.89  & 6.79  & 10.72 & 2.16  & 6.44  & 11.56 & 1.89  & 6.08  & 9.16 \\
          & 20    & 0.00  & 0.49  & 1.84  & 4.81  & 9.26  & 12.60 & 4.32  & 8.21  & 12.13 & 3.74  & 7.16  & 10.43 & 3.74  & 6.98  & 10.43 \\
    \cmidrule( r){2-2}\cmidrule(lr){3-5}\cmidrule(lr){6-8}\cmidrule(lr){9-11}\cmidrule(lr){12-14}\cmidrule(lr){15-17}
    \multicolumn{2}{r}{\textbf{Average\ \ }} & \textbf{0.00} & \textbf{0.62} & \textbf{3.11} & \textbf{4.81} & \textbf{9.55} & \textbf{14.02} & \textbf{1.89} & \textbf{7.76} & \textbf{12.64} & \textbf{2.16} & \textbf{7.24} & \textbf{13.69} & \textbf{1.89} & \textbf{6.90} & \textbf{11.45} \\
    \midrule
          &       & {}   &  &  & &  &  &  &   &   &   &  & & & & \\[-1.5ex]
    8     & 5     & 0.01  & 1.10  & 3.83  & 5.78  & 10.68 & 16.93 & 4.70  & 10.35 & 17.95 & 4.64  & 9.38  & 17.97 & 4.64  & 9.00  & 16.93 \\
          & 10    & 0.01  & 1.15  & 5.44  & 6.35  & 10.58 & 19.82 & 2.90  & 9.48  & 16.39 & 3.46  & 9.02  & 16.39 & 2.90  & 8.46  & 16.39 \\
          & 15    & 0.00  & 1.50  & 6.77  & 5.72  & 10.64 & 17.46 & 4.94  & 9.45  & 17.73 & 4.28  & 9.10  & 17.87 & 4.28  & 8.85  & 17.46 \\
          & 20    & 0.00  & 1.20  & 7.11  & 5.42  & 10.52 & 17.75 & 4.07  & 9.26  & 19.11 & 5.49  & 10.10 & 16.81 & 4.07  & 8.73  & 16.81 \\
\cline{2-17} \noalign{\smallskip}
    \multicolumn{2}{r}{\textbf{Average\ \ }} & \textbf{0.00} & \textbf{1.23} & \textbf{7.11} & \textbf{5.42} & \textbf{10.60} & \textbf{19.82} & \textbf{2.90} & \textbf{9.63} & \textbf{19.11} & \textbf{3.46} & \textbf{9.39} & \textbf{17.97} & \textbf{2.90} & \textbf{8.76} & \textbf{17.46} \\
    \bottomrule
        \end{tabular}}%
    	\caption{Minimum, average (shifted geometric mean) and maximum relative percentage gaps between between lower bounds and the best upper bound.}\label{tab:gaps_LB}
    \end{table}

\begin{table}
\centering
{\small
     \begin{tabular}{@{}cccccccccccccc@{}}    \toprule
    Network & Storage& \multicolumn{3}{c}{\%gap LP-TDIP(RD)  } & \multicolumn{3}{c}{\%gap
      UB-TDIP(RD)   } & \multicolumn{3}{c}{\%gap LP-TDIP(TI) } & \multicolumn{3}{c}{\%gap UB-TDIP(TI) }\\
   & Capacity & \multicolumn{3}{c}{\& Best LB  } & \multicolumn{3}{c}{\& Best LB   } & \multicolumn{3}{c}{\& Best LB } & \multicolumn{3}{c}{\& Best LB } \\
  \cmidrule( r){1-2}\cmidrule(lr){3-5}\cmidrule(lr){6-8}\cmidrule(lr){9-11}\cmidrule(lr){12-14}
		&  & min & avg & max & min & avg & max & min & avg & max & min & avg & max \\ \cmidrule(lr){3-5}\cmidrule(lr){6-8}\cmidrule(lr){9-11}\cmidrule(lr){12-14}
    1	& {5} & 0.00  & 0.15  & 0.36  & 0.00  & 0.10  & 0.27  & 0.00  & 0.14  & 0.36  & 0.00  & 0.09  & 0.36 \\
		& {10} & 0.00  & 0.10  & 0.28  & 0.00  & 0.06  & 0.23  & 0.00  & 0.09  & 0.27  & 0.00  & 0.04  & 0.25 \\
		& {15} & 0.00  & 0.07  & 0.24  & 0.00  & 0.04  & 0.15  & 0.00  & 0.06  & 0.23  & 0.00  & 0.02  & 0.11 \\
		& {20} & 0.00  & 0.05  & 0.21  & 0.00  & 0.03  & 0.11  & 0.00  & 0.05  & 0.21  & 0.00  & 0.01  & 0.03 \\
\cmidrule( r){2-2}\cmidrule(lr){3-5}\cmidrule(lr){6-8}\cmidrule(lr){9-11}\cmidrule(lr){12-14}
    \multicolumn{2}{r}{\textbf{Average\ \ }} & \textbf{0.00} & \textbf{0.09} & \textbf{0.36} & \textbf{0.00} & \textbf{0.06} & \textbf{0.27} & \textbf{0.00} & \textbf{0.09} & \textbf{0.36} & \textbf{0.00} & \textbf{0.04} & \textbf{0.36} \\
\midrule
		& {} &       &       &       &       &       &       &       &       &       &       &       &  \\[-1.5ex]
    2	& {5} & 0.51  & 2.11  & 3.29  & 0.22  & 1.55  & 2.49  & 0.50  & 2.08  & 3.25  & 0.29  & 1.38  & 2.06 \\
		& {10} & 0.51  & 1.98  & 3.12  & 0.22  & 1.47  & 2.37  & 0.49  & 1.95  & 3.10  & 0.28  & 1.32  & 1.93 \\
		& {15} & 0.46  & 1.88  & 3.07  & 0.16  & 1.37  & 2.31  & 0.44  & 1.86  & 3.06  & 0.25  & 1.24  & 1.87 \\
		& {20} & 0.42  & 1.73  & 2.97  & 0.12  & 1.25  & 2.34  & 0.40  & 1.71  & 2.97  & 0.22  & 1.10  & 1.80 \\
\cmidrule( r){2-2}\cmidrule(lr){3-5}\cmidrule(lr){6-8}\cmidrule(lr){9-11}\cmidrule(lr){12-14}
    \multicolumn{2}{r}{\textbf{Average\ \ }} & \textbf{0.42} & \textbf{1.92} & \textbf{3.29} & \textbf{0.12} & \textbf{1.41} & \textbf{2.49} & \textbf{0.40} & \textbf{1.90} & \textbf{3.25} & \textbf{0.22} & \textbf{1.26} & \textbf{2.06} \\
\midrule
		& {} &       &       &       &       &       &       &       &       &       &       &       &  \\[-1.5ex]
    3  	& {5} & 0.00  & 0.00  & 0.00  & 0.00  & 0.00  & 0.00  & 0.00  & 0.00  & 0.00  & 0.00  & 0.00  & 0.00 \\
		& {10} & 0.00  & 0.00  & 0.00  & 0.00  & 0.00  & 0.00  & 0.00  & 0.00  & 0.00  & 0.00  & 0.00  & 0.00 \\
		& {15} & 0.00  & 0.00  & 0.00  & 0.00  & 0.00  & 0.00  & 0.00  & 0.00  & 0.00  & 0.00  & 0.00  & 0.00 \\
		& {20} & 0.00  & 0.00  & 0.00  & 0.00  & 0.00  & 0.00  & 0.00  & 0.00  & 0.00  & 0.00  & 0.00  & 0.00 \\
\cmidrule( r){2-2}\cmidrule(lr){3-5}\cmidrule(lr){6-8}\cmidrule(lr){9-11}\cmidrule(lr){12-14}
    \multicolumn{2}{r}{\textbf{Average\ \ }} & \textbf{0.00} & \textbf{0.00} & \textbf{0.00} & \textbf{0.00} & \textbf{0.00} & \textbf{0.00} & \textbf{0.00} & \textbf{0.00} & \textbf{0.00} & \textbf{0.00} & \textbf{0.00} & \textbf{0.00} \\
\midrule
		& {} &       &       &       &       &       &       &       &       &       &       &       &  \\[-1.5ex]
    4 	& {5} & 0.00  & 0.40  & 1.44  & 0.00  & 0.40  & 1.42  & 0.00  & 0.40  & 1.44  & 0.00  & 0.40  & 1.42 \\
		& {10} & 0.00  & 0.32  & 1.27  & 0.00  & 0.32  & 1.26  & 0.00  & 0.32  & 1.27  & 0.00  & 0.32  & 1.26 \\
		& {15} & 0.00  & 0.21  & 0.77  & 0.00  & 0.21  & 0.74  & 0.00  & 0.21  & 0.76  & 0.00  & 0.21  & 0.74 \\
		& {20} & 0.00  & 0.18  & 0.56  & 0.00  & 0.18  & 0.54  & 0.00  & 0.18  & 0.55  & 0.00  & 0.18  & 0.54 \\
\cmidrule( r){2-2}\cmidrule(lr){3-5}\cmidrule(lr){6-8}\cmidrule(lr){9-11}\cmidrule(lr){12-14}
    \multicolumn{2}{r}{\textbf{Average\ \ }} & \textbf{0.00} & \textbf{0.28} & \textbf{1.44} & \textbf{0.00} & \textbf{0.27} & \textbf{1.42} & \textbf{0.00} & \textbf{0.28} & \textbf{1.44} & \textbf{0.00} & \textbf{0.27} & \textbf{1.42} \\
\midrule
		& {} &       &       &       &       &       &       &       &       &       &       &       &  \\[-1.5ex]
    5 	& {5} & 0.18  & 2.24  & 6.86  & 0.18  & 2.24  & 6.84  & 0.18  & 2.24  & 6.85  & 0.18  & 2.24  & 6.84 \\
		& {10} & 0.26  & 2.07  & 3.77  & 0.26  & 2.06  & 3.76  & 0.26  & 2.07  & 3.77  & 0.26  & 2.06  & 3.76 \\
		& {15} & 0.10  & 1.77  & 3.84  & 0.10  & 1.77  & 3.84  & 0.10  & 1.77  & 3.84  & 0.10  & 1.77  & 3.83 \\
		& {20} & 0.23  & 1.74  & 3.11  & 0.23  & 1.74  & 3.11  & 0.23  & 1.74  & 3.11  & 0.23  & 1.74  & 3.11 \\
\cmidrule( r){2-2}\cmidrule(lr){3-5}\cmidrule(lr){6-8}\cmidrule(lr){9-11}\cmidrule(lr){12-14}
    \multicolumn{2}{r}{\textbf{Average\ \ }} & \textbf{0.10} & \textbf{1.95} & \textbf{6.86} & \textbf{0.10} & \textbf{1.94} & \textbf{6.84} & \textbf{0.10} & \textbf{1.95} & \textbf{6.85} & \textbf{0.10} & \textbf{1.94} & \textbf{6.84} \\
\midrule
		& {} &       &       &       &       &       &       &       &       &       &       &       &  \\[-1.5ex]
    6 	& {5} & 3.97  & 6.78  & 9.39  & 3.67  & 6.47  & 9.28  & 3.96  & 6.76  & 9.36  & 3.72  & 6.50  & 9.27 \\
		& {10} & 3.14  & 6.89  & 10.85 & 2.91  & 6.63  & 10.76 & 3.11  & 6.87  & 10.83 & 2.92  & 6.69  & 10.76 \\
		& {15} & 2.68  & 6.63  & 12.29 & 2.51  & 6.37  & 12.10 & 2.67  & 6.61  & 12.25 & 2.52  & 6.43  & 12.01 \\
		& {20} & 4.53  & 7.20  & 11.76 & 4.29  & 6.94  & 11.55 & 4.53  & 7.18  & 11.72 & 4.40  & 7.02  & 11.54 \\
\cmidrule( r){2-2}\cmidrule(lr){3-5}\cmidrule(lr){6-8}\cmidrule(lr){9-11}\cmidrule(lr){12-14}
    \multicolumn{2}{r}{\textbf{Average\ \ }} & \textbf{2.68} & \textbf{6.87} & \textbf{12.29} & \textbf{2.51} & \textbf{6.60} & \textbf{12.10} & \textbf{2.67} & \textbf{6.85} & \textbf{12.25} & \textbf{2.52} & \textbf{6.66} & \textbf{12.01} \\
\midrule
		& {} &       &       &       &       &       &       &       &       &       &       &       &  \\[-1.5ex]
    7 	& {5} & 0.06  & 0.70  & 2.23  & 0.06  & 0.70  & 2.23  & 0.06  & 0.70  & 2.23  & 0.06  & 0.70  & 2.23 \\
		& {10} & 0.02  & 0.64  & 2.88  & 0.02  & 0.64  & 2.88  & 0.02  & 0.64  & 2.88  & 0.02  & 0.64  & 2.88 \\
		& {15} & 0.00  & 0.65  & 3.11  & 0.00  & 0.65  & 3.11  & 0.00  & 0.65  & 3.11  & 0.00  & 0.65  & 3.11 \\
		& {20} & 0.00  & 0.49  & 1.84  & 0.00  & 0.49  & 1.84  & 0.00  & 0.49  & 1.84  & 0.00  & 0.49  & 1.84 \\
\cmidrule( r){2-2}\cmidrule(lr){3-5}\cmidrule(lr){6-8}\cmidrule(lr){9-11}\cmidrule(lr){12-14}
    \multicolumn{2}{r}{\textbf{Average\ \ }} & \textbf{0.00} & \textbf{0.62} & \textbf{3.11} & \textbf{0.00} & \textbf{0.62} & \textbf{3.11} & \textbf{0.00} & \textbf{0.62} & \textbf{3.11} & \textbf{0.00} & \textbf{0.62} & \textbf{3.11} \\
\midrule
		& {} &       &       &       &       &       &       &       &       &       &       &       &  \\[-1.5ex]
    8 	& {5} & 0.01  & 1.10  & 3.83  & 0.01  & 1.10  & 3.83  & 0.01  & 1.10  & 3.83  & 0.01  & 1.10  & 3.83 \\
		& {10} & 0.01  & 1.15  & 5.44  & 0.01  & 1.15  & 5.44  & 0.01  & 1.15  & 5.44  & 0.01  & 1.15  & 5.44 \\
		& {15} & 0.00  & 1.50  & 6.77  & 0.00  & 1.50  & 6.77  & 0.00  & 1.50  & 6.77  & 0.00  & 1.50  & 6.77 \\
		& {20} & 0.00  & 1.20  & 7.11  & 0.00  & 1.20  & 7.11  & 0.00  & 1.20  & 7.11  & 0.00  & 1.20  & 7.11 \\
\cmidrule( r){2-2}\cmidrule(lr){3-5}\cmidrule(lr){6-8}\cmidrule(lr){9-11}\cmidrule(lr){12-14}
    \multicolumn{2}{r}{\textbf{Average\ \ }} & \textbf{0.00} & \textbf{1.23} & \textbf{7.11} & \textbf{0.00} & \textbf{1.23} & \textbf{7.11} & \textbf{0.00} & \textbf{1.23} & \textbf{7.11} & \textbf{0.00} & \textbf{1.23} & \textbf{7.11} \\
    \bottomrule
    \end{tabular}}
 \caption{Minimum, average (shifted geometric mean) and maximum percentage gaps between upper bounds and the best lower bound.}\label{tab:gaps_UB}
\end{table}

	\begin{sidewaystable}[htbp]
\renewcommand{\arraystretch}{1.2}
  \centering  
\scalebox{0.68}{
    \begin{tabular}{@{}ccc@{\  }c@{\  }ccc@{\  }c@{\  }ccc@{\  }c@{\ }ccc@{\  }c@{\  }ccc@{\  }c@{\   }cc@{}}
    \toprule
    Network & Storage  & \multicolumn{4}{c}{LP-TDIP(RD)} & \multicolumn{4}{c}{UB-TDIP(RD)} & \multicolumn{4}{c}{LP-TDIP(TI)} & \multicolumn{4}{c}{UB-TDIP(TI)} & \multicolumn{4}{c}{LB1} \\
   &  capacity & \multicolumn{4}{c}{} & \multicolumn{4}{c}{} & \multicolumn{4}{c}{} & \multicolumn{4}{c}{} & \multicolumn{4}{c}{} \\

 \cmidrule(lr){3-6}\cmidrule(lr){7-10}\cmidrule(lr){11-14}\cmidrule(lr){15-18}\cmidrule(l ){19-22} 
    & & min   & avg   & max   & \# time lim & min   & avg   & max   & \# time lim & min   & avg
    & max   & \# time lim & min   & avg   & max   & \# time lim & min   & avg   & max   & \#
    time limit \\
 \cmidrule( r){1-2}\cmidrule(lr){3-6}\cmidrule(lr){7-10}\cmidrule(lr){11-14}\cmidrule(lr){15-18}\cmidrule(l ){19-22}
    1     & 5     & 0.39  & 0.68  & 1.14  & 0     & 7.56  & 123.85 & 1932.99 & 0     & 1.33  & 1.78  & 2.29  & 0     & 1108.14 & 4794.57 & 7200  & 6     & 82.83 & 4029.62 & 7200  & 8 \\
          & 10    & 0.43  & 0.61  & 0.76  & 0     & 20.46 & 74.87 & 248.73 & 0     & 1.38  & 1.90  & 2.70  & 0     & 68.37 & 2046.89 & 7200  & 6     & 139.16 & 1375.25 & 7200  & 2 \\
          & 15    & 0.50  & 0.67  & 0.78  & 0     & 9.72  & 52.15 & 139.94 & 0     & 1.37  & 2.06  & 2.90  & 0     & 110.40 & 1661.98 & 7200  & 4     & 80.83 & 854.58 & 7200  & 4 \\
          & 20    & 0.36  & 0.49  & 0.67  & 0     & 8.51  & 40.13 & 103.25 & 0     & 1.30  & 1.86  & 2.29  & 0     & 57.64 & 955.02 & 7200  & 1     & 32.11 & 224.784 & 1078.05 & 0 \\
\cmidrule( r){2-2}\cmidrule(lr){3-6}\cmidrule(lr){7-10}\cmidrule(lr){11-14}\cmidrule(lr){15-18}\cmidrule(l ){19-22} 
    \multicolumn{2}{r}{\textbf{Average\ \ }}  & \textbf{0.36} & \textbf{0.60} & \textbf{1.14} & \textbf{0} & \textbf{7.56} & \textbf{66.37} & \textbf{1932.99} & \textbf{0} & \textbf{1.30} & \textbf{1.90} & \textbf{2.90} & \textbf{0} & \textbf{57.64} & \textbf{1986.65} & \textbf{7200} & \textbf{17} & \textbf{32.11} & \textbf{1015.76} & \textbf{7200} & \textbf{14} \\
\midrule
    2     & 5     & 2.45  & 4.29  & 8.50  & 0     & 7200  & 7200  & 7200  & 10    & 4.88  & 11.44 & 31.07 & 0     & 7200  & 7200  & 7200  & 10    & 7200  & 7200  & 7200  & 10 \\
          & 10    & 1.82  & 3.77  & 7.29  & 0     & 7200  & 7200  & 7200  & 10    & 6.01  & 11.25 & 23.41 & 0     & 7200  & 7200  & 7200  & 10    & 7200  & 7200  & 7200  & 10 \\
          & 15    & 2.25  & 4.18  & 6.95  & 0     & 7200  & 7200  & 7200  & 10    & 5.43  & 11.71 & 21.72 & 0     & 7200  & 7200  & 7200  & 10    & 7200  & 7200  & 7200  & 10 \\
          & 20    & 1.52  & 2.95  & 5.18  & 0     & 7200  & 7200  & 7200  & 10    & 5.42  & 11.42 & 19.25 & 0     & 7200  & 7200  & 7200  & 10    & 7200  & 7200  & 7200  & 10 \\
\cmidrule( r){2-2}\cmidrule(lr){3-6}\cmidrule(lr){7-10}\cmidrule(lr){11-14}\cmidrule(lr){15-18}\cmidrule(l ){19-22}
    \multicolumn{2}{r}{\textbf{Average\ \ }}  & \textbf{1.52} & \textbf{3.76} & \textbf{8.50} & \textbf{0} & \textbf{7200} & \textbf{7200} & \textbf{7200} & \textbf{40} & \textbf{4.88} & \textbf{11.46} & \textbf{31.07} & \textbf{0} & \textbf{7200} & \textbf{7200} & \textbf{7200} & \textbf{40} & \textbf{7200} & \textbf{7200} & \textbf{7200} & \textbf{40} \\
\midrule
    3     & 5     & 1.64  & 1.89  & 2.17  & 0     & 19.67 & 46.16 & 115.63 & 0     & 2.04  & 2.37  & 2.67  & 0     & 48.19 & 245.75 & 431.93 & 0     & 46.94 & 92.138 & 128.93 & 0 \\
          & 10    & 1.33  & 1.60  & 2.15  & 0     & 18.62 & 31.33 & 53.32 & 0     & 2.35  & 2.79  & 3.21  & 0     & 117.78 & 322.78 & 757.44 & 0     & 53.21 & 80.5027 & 150.93 & 0 \\
          & 15    & 1.75  & 2.12  & 2.91  & 0     & 18.87 & 41.03 & 99.18 & 0     & 2.61  & 2.85  & 3.37  & 0     & 109.68 & 226.16 & 533.15 & 0     & 55.05 & 82.5626 & 117.37 & 0 \\
          & 20    & 1.26  & 1.52  & 1.86  & 0     & 18.36 & 38.86 & 124.35 & 0     & 2.25  & 2.63  & 2.97  & 0     & 111.99 & 228.64 & 875.31 & 0     & 49.73 & 90.7722 & 162.99 & 0 \\
\cmidrule( r){2-2}\cmidrule(lr){3-6}\cmidrule(lr){7-10}\cmidrule(lr){11-14}\cmidrule(lr){15-18}\cmidrule(l ){19-22}
    \multicolumn{2}{r}{\textbf{Average\ \ }}  & \textbf{1.26} & \textbf{1.77} & \textbf{2.91} & \textbf{0} & \textbf{18.36} & \textbf{38.97} & \textbf{124.35} & \textbf{0} & \textbf{2.04} & \textbf{2.65} & \textbf{3.37} & \textbf{0} & \textbf{48.19} & \textbf{253.07} & \textbf{875.31} & \textbf{0} & \textbf{46.94} & \textbf{86.35} & \textbf{162.99} & \textbf{0} \\
\midrule
    4     & 5     & 6.02  & 15.31 & 36.42 & 0     & 3305.32 & 6660.71 & 7200  & 9     & 7.75  & 17.25 & 47.90 & 0     & 2659.50 & 6517.47 & 7200  & 9     & 7200  & 7200  & 7200  & 10 \\
          & 10    & 8.10  & 16.17 & 42.79 & 0     & 1664.31 & 5705.44 & 7200  & 8     & 7.62  & 17.40 & 40.44 & 0     & 3477.27 & 6375.37 & 7200  & 8     & 1441.61 & 6130.33 & 7200  & 9 \\
          & 15    & 6.52  & 17.73 & 41.12 & 0     & 418.42 & 5115.78 & 7200  & 6     & 7.76  & 17.13 & 40.41 & 0     & 5244.40 & 6905.67 & 7200  & 8     & 2007  & 6336.57 & 7200  & 9 \\
          & 20    & 5.28  & 11.70 & 22.48 & 0     & 643.63 & 4525.00 & 7200  & 6     & 8.04  & 15.25 & 32.50 & 0     & 4296.82 & 6487.23 & 7200  & 7     & 1173.74 & 4926.81 & 7200  & 7 \\
\cmidrule( r){2-2}\cmidrule(lr){3-6}\cmidrule(lr){7-10}\cmidrule(lr){11-14}\cmidrule(lr){15-18}\cmidrule(l ){19-22}
    \multicolumn{2}{r}{\textbf{Average\ \ }}  & \textbf{5.28} & \textbf{15.05} & \textbf{42.79} & \textbf{0} & \textbf{418.42} & \textbf{5446.09} & \textbf{7200} & \textbf{29} & \textbf{7.62} & \textbf{16.73} & \textbf{47.90} & \textbf{0} & \textbf{2659.50} & \textbf{6568.44} & \textbf{7200} & \textbf{32} & \textbf{1173.74} & \textbf{6092.69} & \textbf{7200} & \textbf{35} \\
\midrule
    5     & 5     & 42.32 & 166.73 & 293.56 & 0     & 7200  & 7200  & 7200  & 10    & 40.67 & 168.57 & 310.82 & 0     & 7200  & 7200  & 7200  & 10    & 7200  & 7200  & 7200  & 10 \\
          & 10    & 33.74 & 161.50 & 272.81 & 0     & 7200  & 7200  & 7200  & 10    & 60.37 & 191.09 & 293.68 & 0     & 7200  & 7200  & 7200  & 10    & 7200  & 7200  & 7200  & 10 \\
          & 15    & 41.21 & 150.15 & 255.19 & 0     & 7200  & 7200  & 7200  & 10    & 36.96 & 157.53 & 260.70 & 0     & 7200  & 7200  & 7200  & 10    & 7200  & 7200  & 7200  & 10 \\
          & 20    & 31.58 & 111.14 & 210.31 & 0     & 7200  & 7200  & 7200  & 10    & 34.51 & 128.53 & 231.28 & 0     & 7200  & 7200  & 7200  & 10    & 7200  & 7200  & 7200  & 10 \\
\cmidrule( r){2-2}\cmidrule(lr){3-6}\cmidrule(lr){7-10}\cmidrule(lr){11-14}\cmidrule(lr){15-18}\cmidrule(l ){19-22}
    \multicolumn{2}{r}{\textbf{Average\ \ }}  & \textbf{31.58} & \textbf{145.59} & \textbf{293.56} & \textbf{0} & \textbf{7200} & \textbf{7200} & \textbf{7200} & \textbf{40} & \textbf{34.51} & \textbf{159.81} & \textbf{310.82} & \textbf{0} & \textbf{7200} & \textbf{7200} & \textbf{7200} & \textbf{40} & \textbf{7200} & \textbf{7200} & \textbf{7200} & \textbf{40} \\
\midrule
    6     & 5     & 47.21 & 152.84 & 296.24 & 0     & 7200  & 7200  & 7200  & 10    & 55.66 & 170.61 & 325.87 & 0     & 7200  & 7200  & 7200  & 10    & 7200  & 7200  & 7200  & 10 \\
          & 10    & 37.88 & 121.54 & 280.85 & 0     & 7200  & 7200  & 7200  & 10    & 79.33 & 224.92 & 349.82 & 0     & 7200  & 7200  & 7200  & 10    & 7200  & 7200  & 7200  & 10 \\
          & 15    & 50.38 & 139.22 & 267.68 & 0     & 7200  & 7200  & 7200  & 10    & 65.39 & 180.17 & 310.39 & 0     & 7200  & 7200  & 7200  & 10    & 7200  & 7200  & 7200  & 10 \\
          & 20    & 35.39 & 112.54 & 295.50 & 0     & 7200  & 7200  & 7200  & 10    & 63.16 & 168.17 & 302.63 & 0     & 7200  & 7200  & 7200  & 10    & 7200  & 7200  & 7200  & 10 \\
\cmidrule( r){2-2}\cmidrule(lr){3-6}\cmidrule(lr){7-10}\cmidrule(lr){11-14}\cmidrule(lr){15-18}\cmidrule(l ){19-22}
    \multicolumn{2}{r}{\textbf{Average\ \ }}  & \textbf{35.39} & \textbf{130.61} & \textbf{296.24} & \textbf{0} & \textbf{7200} & \textbf{7200} & \textbf{7200} & \textbf{40} & \textbf{55.66} & \textbf{184.66} & \textbf{349.82} & \textbf{0} & \textbf{7200} & \textbf{7200} & \textbf{7200} & \textbf{40} & \textbf{7200} & \textbf{7200} & \textbf{7200} & \textbf{40} \\
\midrule
    7     & 5     & 56.43 & 99.52 & 272.67 & 0     & 7200  & 7200  & 7200  & 10    & 42.19 & 80.73 & 250.25 & 0     & 7200  & 7200  & 7200  & 10    & 7200  & 7200  & 7200  & 10 \\
          & 10    & 45.29 & 84.83 & 258.97 & 0     & 7200  & 7200  & 7200  & 10    & 43.08 & 70.31 & 221.18 & 0     & 7200  & 7200  & 7200  & 10    & 7200  & 7200  & 7200  & 10 \\
          & 15    & 51.39 & 75.08 & 213.98 & 0     & 7200  & 7200  & 7200  & 10    & 39.34 & 62.40 & 218.22 & 0     & 7200  & 7200  & 7200  & 10    & 7156.27 & 7200  & 7200  & 9 \\
          & 20    & 36.98 & 63.82 & 190.76 & 0     & 7200  & 7200  & 7200  & 10    & 37.95 & 59.63 & 155.15 & 0     & 7200  & 7200  & 7200  & 10    & 7081.23 & 7200  & 7200  & 9 \\
\cmidrule( r){2-2}\cmidrule(lr){3-6}\cmidrule(lr){7-10}\cmidrule(lr){11-14}\cmidrule(lr){15-18}\cmidrule(l ){19-22}
    \multicolumn{2}{r}{\textbf{Average\ \ }}  & \textbf{36.98} & \textbf{79.75} & \textbf{272.67} & \textbf{0} & \textbf{7200} & \textbf{7200} & \textbf{7200} & \textbf{40} & \textbf{37.95} & \textbf{67.79} & \textbf{250.25} & \textbf{} & \textbf{7200} & \textbf{7200} & \textbf{7200} & \textbf{40} & \textbf{7081.23} & \textbf{7200} & \textbf{7200} & \textbf{38} \\
\midrule
    8     & 5     & 100.97 & 165.81 & 283.22 & 0     & 7200  & 7200  & 7200  & 10    & 61.50 & 133.69 & 320.92 & 0     & 7200  & 7200  & 7200  & 10    & 6785.55 & 7200  & 7200  & 9 \\
          & 10    & 100.49 & 166.42 & 281.70 & 0     & 7200  & 7200  & 7200  & 10    & 68.03 & 144.74 & 213.20 & 0     & 7200  & 7200  & 7200  & 10    & 7098.63 & 7200  & 7200  & 9 \\
          & 15    & 90.45 & 136.79 & 207.05 & 0     & 7200  & 7200  & 7200  & 10    & 65.40 & 107.90 & 231.74 & 0     & 7200  & 7200  & 7200  & 10    & 6440.05 & 7120.13 & 7200  & 9 \\
          & 20    & 63.72 & 129.98 & 273.14 & 0     & 7200  & 7200  & 7200  & 10    & 62.40 & 100.90 & 258.50 & 0     & 7200  & 7200  & 7200  & 10    & 6538  & 7130.89 & 7200  & 9 \\
\cmidrule( r){2-2}\cmidrule(lr){3-6}\cmidrule(lr){7-10}\cmidrule(lr){11-14}\cmidrule(lr){15-18}\cmidrule(l ){19-22}
    \multicolumn{2}{r}{\textbf{Average\ \ }}  & \textbf{63.72} & \textbf{148.83} & \textbf{283.22} & \textbf{0} & \textbf{7200} & \textbf{7200} & \textbf{7200} & \textbf{40} & \textbf{61.50} & \textbf{120.47} & \textbf{320.92} & \textbf{0} & \textbf{7200} & \textbf{7200} & \textbf{7200} & \textbf{40} & \textbf{6440.05} & \textbf{7162.66} & \textbf{7200} & \textbf{36} \\
    \bottomrule
    \end{tabular} }%
  \caption{Minimum, average (geometric mean) and maximum run times to obtain upper and lower bounds, in seconds. }\label{tab:run_times}%
\end{sidewaystable}%

\begin{table}[htbp]
\renewcommand{\arraystretch}{1.2}
  \centering
    \begin{tabular}{@{}ccccc@{}} \toprule
    Year  & \% gap LP-TDIP(RD)  & \% gap UB-TDIP(RD)  & \% gap LP-TDIP(TI)  & \% gap UB-TDIP(TI)  \\
    & \& Best LB & \& Best LB & \& Best LB & \& Best LB \\ \midrule
    2010  & 8.49  & 8.04  & 8.43  & \textbf{7.64} \\
    2011  & 5.82  & \textbf{5.65}  & 5.81  & 5.74 \\
    \bottomrule
    \end{tabular}%
    \caption{Percentage gaps between upper bounds and the best lower bound for the HVCC instances.}\label{tab:gaps_UB_HVCC}%
\end{table}%

\begin{table}[htbp]
\renewcommand{\arraystretch}{1.2}
  \centering
    \begin{tabular}{@{}cS[table-format=2.2]S[table-format=2.2]S[table-format=2.2]S[table-format=2.2]S[table-format=2.2]@{}} \toprule
    Year  & {\% Gap Best UB} & {\% Gap Best UB} & {\% Gap Best UB} & {\% Gap Best UB} & {\% Gap Best UB}\\
          & {\& LB1} & {\& CoM-LP1800} & {\& CoM-FS1800} & {\& Proj-LP1800} & {\& Proj-FS1800} \\ \midrule
    2010  & 7.65  & 10.77 & $\mathbf{7.64}$  & 9.00  & 7.77 \\
    2011  & 11.55 & 7.88  & $\mathbf{5.65}$  & 6.90  & 5.68 \\
    \bottomrule
    \end{tabular}%
  \caption{Percentage gaps between lower bounds and the best upper bound for the HVCC instances.}\label{tab:gaps_LB_HVCC}%
\end{table}%

\begin{table}[htbp]
\renewcommand{\arraystretch}{1.2}
  \centering
    \begin{tabular}{@{}ccccc@{}} \toprule
    Year  &  LP-TDIP(RD)  &  UB-TDIP(RD)  & LP-TDIP(TI)  & UB-TDIP(TI)  \\
    \midrule
    2010  & 50.46 & 14400.00 & 1537.14 & 14400.00 \\
    2011  & 69.20 & 14400.00 & 2558.93 & 14400.00 \\
    \bottomrule
    \end{tabular}%
    \caption{Run times of TDIP LP relaxations and MIPs for the HVCC instances, in seconds.}\label{tab:run_times_HVCC}%
\end{table}

\section{Extensions toward practical application}\label{sec:practical}

Although our focus in this paper is on a problem motivated by practical applications in maintenance scheduling, in order to obtain insights about the nature of the problem and its theoretical properties, we have studied a variant that is simpler than one that might be encountered in practice. In this section, we mention a few extensions that we have encountered, or that might naturally arise, in practice, and briefly indicate how they might be represented in our model.

In the HVCC application that motivated this research, (described, for example, in
\cite{boland2012mixed}), precedence relations between jobs were encountered. These can easily
modeled linearly in the CTIP model, for example, with the constraint $w_{ai} + w_{a'j} \leqslant 1$
for all $j \leqslant i$ when the job on arc $a$ must be completed before the one on arc $a'$ can start. In the TDIP model, these
relations can be approximated, for example, with the constraints
\[\sum_{i'\in\mathcal S_a\,:\,i'\geqslant i}y_{ai'}+\sum_{j\in\mathcal S_{a'}\,:\,t_j\leqslant
  t_{i-1}+p_a}y_{a'j}\leqslant 1\qquad\text{for all }i\in\mathcal S_a,\]
which capture that if the job on arc $a$ starts at time $t_{i-1}$ or later then the job on arc $a'$
cannot start before time $t_{i-1}+p_a$.


Incompatible sets of maintenance jobs, at most one of which could be in progress at any one time, were also encountered in the HVCC application. This is readily represented in the CTIP model with the constraint $\sum_{a\in{\cal C}} w_{ai} \leqslant 1$, and in the TDIP model with $\sum_{a\in {\cal C}} z_{ai} \leqslant 1$, where $\cal C$ is a set of arcs with mutually incompatible jobs.

In practical settings, it may also be the case that resources required to carry out maintenance, such as work crews, or equipment, are limited. In the HVCC setting, such limitations did indeed occur, but their nature led them to be handled by a combination of the release dates and deadlines, precedence constraints, and incompatible job sets. For example, incompatible jobs sets may consist of jobs that require a special type of maintenance equipment, only one of which is available. These latter constraints can readily be generalized to the case of a limited number of maintenance jobs that could occur at any one time: in the constraints given in the last paragraph, the right-hand sides of $1$ can simply be replaced by the required limit. Clearly, in general, there are many possible ways in which resources may constrain a maintenance schedule, and the precise nature of these will be application-dependent. 

Another natural consideration from an application point of view is the possibility of maintenance on nodes, where a job on node $v$ prevents flow through this node for the corresponding time period. This can be easily captured by our model using the standard node-splitting device. 

For a non-storage node
$v\in V\setminus W$ with a maintenance job, we replace $v$ by two nodes $v'$ and $v''$, connected by an
arc $(v',v'')$ whose capacity is
\begin{equation}\label{eq:node_splitting}
	\min\left\{\sum_{a\in\delta^{\text{in}}(v)}u_a,\quad \sum_{a\in\delta^{\text{out}}(v)}u_a\right\} 
\end{equation}
while every arc $a=(w,v)$ in the original network is replaced
by an arc $(w,v')$ of capacity $u_a$, and every arc $a=(v,w)$ in the original network is replaced
by an arc $(v'',w)$ of capacity $u_a$. Then the job on node $v$ is equivalent to a job on the arc
$(v',v'')$ with the same release date, deadline, and processing time.

For storage nodes the situation is a little bit more complicated. One could think of situations where
a maintenance job blocks only inbound flow, only outbound flows or both. To capture this in full
generality we can replace a storage node $v\in W$ by three nodes $v'$, $v''$ and $v'''$ with arcs
$(v',v'')$ and $(v'',v''')$, both with capacity~\eqref{eq:node_splitting}. In the new network $v''$
is a storage node with capacity $u_v$, while $v'$ and $v'''$ are non-storage nodes. Every arc $a=(w,v)$ in the original network is replaced
by an arc $(w,v')$ of capacity $u_a$, and every arc $a=(v,w)$ in the original network is replaced
by an arc $(v''',w)$ of capacity $u_a$. A job on node $v$ that blocks only inbound flow can be
represented by a job on the arc $(v',v'')$, a job on node $v$ that blocks only outbound flow can be
represented by a job on the arc $(v'',v''')$, and a job on node $v$ that blocks both inbound and outbound flow can be
represented by two jobs on arcs $(v',v'')$ and $(v'',v''')$ with the additional constraint that
these two jobs have to be processed at the same time. This can be modeled in the CTIP model by asking that $w_{ai} = w_{a'i}$ for all $i$, and in the TDIP model by asking that $z_{ai} = z_{a'i}$ for all $i$, where $a$ and $a'$ are the arcs on which the two jobs must be processed at the same time.

\section{Conclusions and future work}

Our results immediately suggest two directions in which further investigation is warranted. First, TDIP offers enormous flexibility in the choice of discretization, so the relationship between the quality of the bounds produced by TDIP formulations and the granularity of the discretization used needs to be better understood. The nature of the discretization, for example, whether it is regular or irregular, and whether or not it contains the job release dates and due dates, may also impact performance of the formulation: this, too, needs to be better understood. Second, the CoM and Proj heuristics themselves are computationally very cheap, but their quality depends on the time at which either LP or integer feasible solutions are extracted from the TDIP MIP solution process. The relationships between the instance parameters, this time, and the resulting solution quality need to be better understood and quantified. These are both directions of future study.

\section*{Acknowledgements} We thank the Australian Research Council (ARC), the Hunter Valley Coal Chain Coordinator (HVCCC) and Triple Point Technologies for their generous research funding; this research was supported by ARC Linkage grant LP110200524. We also thank the University of Newcastle for funding the PhD scholarship of our co-author, Simranjit Kaur, and thank the University of Delhi, in particular the Sri Guru Tegh Bahadur Khalsa College, for their ongoing support of her PhD studies.

\FloatBarrier

\clearpage

\appendix
\section{Shifting jobs}\label{app:shifting}
In this appendix we assume that there is no storage, i.e., $W=\varnothing$. Consider an optimal schedule $\vect t^*=(t^*_{a})_{a\in A_1}$ with associated time discretization $0=t_0<t_1<\cdots<t_n=T$. Let $F_i$ be the value of a maximum flow in the network available in time interval $[t_{i-1},t_i]$, so that the objective value for the solution $\vect t^*$ is
\[\val(\vect t^*)=\sum_{i=1}^n(t_i-t_{i-1})F_i.\]
For an index set $I\subseteq\{0,1,\ldots,n\}$, let $\mathcal T_I=\{t_i\ :\ i\in I\}$ be the set of corresponding time points. Conversely, for a set $\mathcal T\subseteq\{t_0,\ldots,t_n\}$, let $I_{\mathcal T}=\{i\ :\ t_i\in\mathcal T\}$ be the corresponding index set.
We define the \emph{closure} of $\mathcal T\subseteq\{t_0,\ldots,t_n\}$ to be the set
\[\cl(\mathcal T)=\{t^*_{a},\,t^*_{a}+p_{a}\ :\ a\in A_1\text{ with }t^*_{a}\in\mathcal T\text{ or }t^*_{a}+p_{a}\in\mathcal T\},\]
and the closure of an index set $I$ is $\cl(I)=I_{\cl(\mathcal T_I)}$.
The set $I$ is called \emph{closed} if $\cl(I)=I$. A closed set $I$ is called \emph{free} if
\[\forall a\in A_1\quad t^*_{a}\in\mathcal T_I\implies t^*_{a}\not\in\{r_{a},d_{a}-p_{a}\}.\]
In other words, if $I$ is a free closed index set, then there is an $\varepsilon>0$ such that all jobs whose start and completion times lie in $\mathcal T_I$ can be moved by $\pm\varepsilon$ to obtain two other feasible solutions which we denote by $\vect t^*(I,\varepsilon)$ and $\vect t^*(I,-\varepsilon)$. More precisely,
\[t^*(I,\pm\varepsilon)_j=
\begin{cases}
  t^*_{a}\pm\varepsilon & \text{if }t^*_{a}\in\mathcal T_I, \\
  t^*_{a} & \text{if }t^*_{a}\not\in\mathcal T_I.
\end{cases}
\]

The value $\varepsilon$ can also be chosen sufficiently small, so that the shift affects only the lengths of the time intervals but not the combinations of available arcs. This can be achieved by choosing
\[\varepsilon\leqslant\min\{\min\{t_i-t_{i-1}\ :\ i\in I,\ i-1\not\in I\},\ \min\{t_{i+1}-t_{i}\ :\ i\in I,\ i+1\not\in I\}\}.\]
Let $I^+=\{i\in I\ :\ i+1\not\in I\}$ and $I^-=\{i\in I\ :\ i-1\not\in I\}$. Then the objective values for the solutions $\vect t^*(\varepsilon)$ and $\vect t^*(-\varepsilon)$ are
\begin{align*}
\val(\vect t^*(I,\varepsilon)) &= \val(\vect t^*) -\varepsilon\left(\sum_{i\in I^+}F_{i+1}-\sum_{i\in I^-}F_i\right), \\
\val(\vect t^*(I,-\varepsilon)) &= \val(\vect t^*) +\varepsilon\left(\sum_{i\in I^+}F_{i+1}-\sum_{i\in I^-}F_i\right).
\end{align*}
The optimality of $\vect t^*$ implies $\val(\vect t^*(I,\varepsilon))=\val(\vect t^*(I,-\varepsilon))=\val(\vect t^*)$, hence we have proved the following lemma.
\begin{lemma}\label{lem:shifting}
Suppose $W=\varnothing$, $\vect t^*$ is an optimal solution, and $I$ is a free closed index set. Then there is an $\varepsilon>0$ such that the two solutions $\vect t^*(I,\varepsilon)$ and $\vect t^*(I,-\varepsilon)$ are also optimal solutions.
\end{lemma}

\end{document}